\documentclass[final,leqno,onefignum,onetabnum]{siamltex1213}

\usepackage{color}
\usepackage{amsmath,amsfonts,amssymb,bm,mathtools} %
\usepackage{graphicx,psfrag,subfigure}                    %
\usepackage{stmaryrd}
\usepackage{enumitem}
\usepackage{epstopdf}
\usepackage{mathrsfs}                           %
\usepackage{multirow}

\allowdisplaybreaks

\renewcommand{\d}{\mathrm{d}}
\newcommand{\e}{\mathrm{e}}

\newcommand{\R}{\mathbb{R}}
\newcommand{\C}{\mathbb{C}}

\newcommand{\Z}{\mathbb{Z}}

\newcommand{\bi}{{\bm{i}}}
\newcommand{\bj}{{\bm{j}}}

\newcommand{\bs}{\bm{s}}

\newcommand{\bx}{\bm{x}}

\newcommand{\bo}{\bm{0}}

\newcommand{\hI}{\mathcal{I}}
\newcommand{\hL}{\mathcal{L}}

\newcommand{\ds}{\displaystyle}
\newcommand{\eps}{\varepsilon}
\newcommand{\daoshu}[2]{\dfrac{\d #1}{\d #2}}
\newcommand{\piandao}[2]{\dfrac{\partial #1}{\partial #2}}

\newcommand{\Cp}{C_\text{\rm per}}
\newtheorem{example}{Example}[section]
\newtheorem{remark}{Remark}[section]
\title{Maximum principle preserving exponential time differencing schemes for the nonlocal Allen-Cahn equation}

\author{
Qiang Du\footnotemark[2]
\and Lili Ju\footnotemark[3]
\and Xiao Li\footnotemark[4]
\and Zhonghua Qiao\footnotemark[5]
}

\begin{document}

\maketitle
\slugger{mms}{xxxx}{xx}{x}{x--x}

\renewcommand{\thefootnote}{\fnsymbol{footnote}}

\footnotetext[2]
{Department of Applied Physics and Applied Mathematics, Columbia University, New York, NY 10027, USA (qd2125@columbia.edu).
Q. Du's work is partially supported by US National Science Foundation grant DMS-1558744,
US AFOSR MURI Center for Material Failure Prediction Through Peridynamics,
and US Army Research Office MURI grant W911NF-15-1-0562.}
\footnotetext[3]
{Department of Mathematics, University of South Carolina, Columbia, SC 29208, USA (ju@math.sc.edu).
L. Ju's work is partially supported by US National Science Foundation grant DMS-1521965
and US Department of Energy grants DE-SC0008087-ER65393 and DE-SC0016540.}
\footnotetext[4]
{Applied and Computational Mathematics Division, Beijing Computational Science Research Center,
Beijing 100193, China (xiaoli@csrc.ac.cn).
X. Li's work is partially supported by China Postdoctoral Science Foundation grant 2017M610748.}
\footnotetext[5]
{Department of Applied Mathematics, The Hong Kong Polytechnic University, Hung Hom, Kowloon, Hong Kong (zhonghua.qiao@polyu.edu.hk).
Z. Qiao's work is partially supported by the Hong Kong Research Council GRF grants 15302214 and 15325816
and the Hong Kong Polytechnic University fund 1-ZE33.}

\renewcommand{\thefootnote}{\arabic{footnote}}

\begin{abstract}
The nonlocal Allen-Cahn (NAC) equation is a generalization of the classic Allen-Cahn equation
by replacing the Laplacian with a parameterized nonlocal diffusion operator,
and satisfies the maximum principle as its local counterpart.
In this paper, we develop and analyze first and second order exponential time differencing (ETD) schemes
for solving the NAC equation, which unconditionally preserve the discrete maximum principle.
The fully discrete numerical schemes are obtained by applying the stabilized ETD approximations
for time integration with the quadrature-based finite difference discretization in space.
We derive their respective optimal maximum-norm error estimates
and further show that the proposed schemes are asymptotically compatible,
i.e., the approximate solutions always converge to the classic Allen-Cahn solution
when the horizon, the spatial mesh size and the time step size go to zero.
We also prove that the schemes are energy stable in the discrete sense.
Various experiments are performed to verify these theoretical results and to investigate
numerically the relation between the discontinuities and the nonlocal parameters.
\end{abstract}



\pagestyle{myheadings}
\thispagestyle{plain}
\markboth{QIANG DU, LILI JU, XIAO LI, AND ZHONGHUA QIAO}
{Maximum Principle Preserving ETD Schemes for the Nonlocal Allen-Cahn Equation}

\section{Introduction}

In this paper,
we consider numerical solutions of the initial-boundary-value problem of the nonlocal Allen-Cahn (NAC) equation as follows:
\begin{subequations}
\label{nonlocalAC}
\begin{align}
u_t-\eps^2\hL_\delta u+u^3-u=0, & \qquad\bx\in\Omega,\ t\in(0,T),\label{nonlocalAC1}\\
\text{$u(\cdot,t)$ is $\Omega$-periodic}, & \qquad t\in[0,T],\\
u(\bx,0)=u_0(\bx), & \qquad\bx\in\overline{\Omega},
\end{align}
\end{subequations}
where $u(\bx,t)$ denotes the unknown function, $\Omega=(0,X)^d$ is a hypercube domain in $\R^d$,
$\eps>0$ is an interfacial parameter, and $\hL_\delta$ is a nonlocal operator,
parameterized by the positive horizon parameter $\delta> 0$ measuring the range of nonlocal interactions.
Assume that $\hL_\delta$ is
defined by
\begin{equation}
\label{nonlocalopt_def}
\hL_\delta u(\bx)=\frac{1}{2}\int_{B_\delta(\bo)}\rho_\delta(|\bs|)\big(u(\bx+\bs)+u(\bx-\bs)-2u(\bx)\big)\,\d\bs,\qquad\bx\in\Omega
\end{equation}
with $B_\delta(\bo)$ denoting the ball in $\R^d$ centered at the origin with the radius $\delta$
and $\rho_\delta:[0,\delta]\to\R$ being  a nonnegative kernel function.
To enforce the consistency, as $\delta\to0$,
of the nonlocal operator $\hL_\delta$ with the standard Laplacian operator $\hL_0:=\Delta$,
we further assume the kernel $\rho_\delta$ satisfies
$$\int_0^\delta r^{1+d}\rho_\delta(r)\,\d r=\frac{2d}{S_d}$$
with $S_d$ being the area of the unit sphere in $\R^d$,
or equivalently,
\begin{equation}
\label{kernel_2ndmoment}
\int_{B_\delta(\bo)}|\bs|^2\rho_\delta(|\bs|)\,\d\bs=2d.
\end{equation}
Note that \eqref{kernel_2ndmoment} also means that the kernel $\rho_\delta$ has a finite second order moment.
The continuum property of the nonlocal operator $\hL_\delta$ gives \cite{Du16hbk4,DuGuLeZh12}
\begin{equation}
\label{continuous_consistency}
\max_{\bx\in\Omega}|\hL_\delta u(\bx)-\hL_0u(\bx)|\le C\delta^2\|u\|_{C^4},\qquad\forall\,u\in\Cp^4(\overline{\Omega}),
\end{equation}
where $C>0$ is a constant independent of $\delta$.
The local limit of the NAC problem \eqref{nonlocalAC} is exactly
the classic (local) Allen-Cahn (LAC) equation taking the following form:
\begin{subequations}
\label{localAC}
\begin{align}
\varphi_t-\eps^2\hL_0\varphi+\varphi^3-\varphi=0, & \qquad\bx\in\Omega,\ t\in(0,T),\\
\text{$\varphi(\cdot,t)$ is $\Omega$-periodic}, & \qquad t\in[0,T],\\
\varphi(\bx,0)=\varphi_0(\bx), & \qquad\bx\in\overline{\Omega}.
\end{align}
\end{subequations}
The LAC equation \eqref{localAC} is a well-known phase field model
used to describe the motion of anti-phase boundaries in crystalline solids \cite{AlCa79}.

In recent years, the nonlocal models involving the nonlocal operator \eqref{nonlocalopt_def},
such as the NAC equation \eqref{nonlocalAC},
have appeared in a variety of applications ranging from physics, materials science to finance and image processing,
for instance, phase transition \cite{Bates06,Fife03}, peridynamics continuum theory \cite{Silling00,SiLe10},
image analyses \cite{GaGa05,GiOs08}, and nonlocal heat conduction \cite{BoDu10}.
Rigorous mathematical analysis of nonlocal models can be found in the literatures, e.g., \cite{AnMaRoTo10,Bates06,DuZh11},
and a more systematic mathematical framework of nonlocal problems was developed in \cite{DuGuLeZh12,DuGuLeZh13}
in parallel to the analysis for the classic partial differential equations.
Since the exact/analytic solutions of these nonlocal models are usually not available,
numerical methods play an important role in studying these models.
Bates et al. \cite{BaBrHa09} considered a finite difference discretization of the NAC equation with an integrable kernel
and developed an $L^\infty$ stable and convergent numerical scheme by treating the nonlinear and nonlocal terms explicitly.
A similar technique was applied on the NAC-type problem coupled with a heat equation
and an $L^\infty$ stable and convergent numerical scheme was obtained \cite{ArBrHa10}.
For the nonlocal diffusion models with more general kernels and variable boundary conditions,
finite difference and finite element approximations were addressed in \cite{DuTaTiYa18,TaTiDu17,TiDu13,ZhDu10}.
To illustrate the limiting behaviors of the numerical solution of the nonlocal model
to the exact solution of the corresponding local counterpart,
Tian and Du proposed in \cite{TiDu14} the concept of \emph{asymptotic compatibility},
and the spectral-Galerkin approximation of the NAC equation was then proved to be asymptotically compatible in \cite{DuYa16}.
The convergence of asymptotically compatible schemes is insensitive to the choices of modeling and discretization parameters
so that such schemes provide robust numerical approximations of nonlocal models.

As a nonlocal analogue of the LAC equation \eqref{localAC},
the NAC equation \eqref{nonlocalAC} possesses some similar properties.
First, it can be shown that the NAC equation \eqref{nonlocalAC} satisfies a \emph{maximum principle}:
if the initial value and the boundary conditions are bounded by $1$,
then the entire solution is also bounded by $1$, i.e.,
$$\|u(\cdot,t)\|_{L^\infty}\le1,\quad\forall\,t>0.$$
Second, as a phase field type model,
the NAC equation \eqref{nonlocalAC1} can be viewed as an $L^2$ gradient flow with respect to the energy functional
\begin{equation}
\label{nonlocalAC_energy}
E(u)=\int_\Omega\Big(\frac{1}{4}(u^2(\bx)-1)^2-\frac{\eps^2}{2}u(\bx)\hL_\delta u(\bx)\Big)\,\d\bx,
\end{equation}
and thus, the solution $u$ to \eqref{nonlocalAC} decreases the energy \eqref{nonlocalAC_energy} in time, that is,
\begin{equation}
\label{energy_decay}
E(u(t_2))\le E(u(t_1)),\quad\forall\,t_2\ge t_1\ge0,
\end{equation}
which is often called the energy dissipation law.
Such two properties are important in the study of the stability of the solution to \eqref{nonlocalAC},
and whether they could be inherited in the discrete level is a significant issue in numerical simulations.
A major objective of this work is
develop maximum principle preserving and energy stable numerical schemes for approximating the NAC equation \eqref{nonlocalAC}.

Energy stability has been widely investigated for numerical schemes of  classic PDE-based phase field models,
such as convex splitting schemes \cite{QiSu14,WaWiLo09}, stabilized schemes \cite{ShYa10,XuTa06},
invariant energy quadratization methods \cite{Yang16,YaHa17} and so on.
It is interesting to study whether similar analysis can be applied
to nonlocal phase field models due to the lack of the high-order diffusion term.
Guan et al. \cite{GuWaWi14} constructed a convex splitting scheme for the nonlocal Cahn-Hilliard equation
by treating the nonlinear term implicitly and setting the nonlocal term into the explicit part.
Their scheme allows one to evaluate the nonlocal term explicitly only once at each time step,
but the nonlinear iterations are still inevitable.
In order to avoid the nonlinear iterations, a stabilized scheme was
the linear stabilization strategy was adopted in \cite{DuJuLiQi18} to develop the stabilized linear schemes
which can be solved efficiently by using the fast Fourier transform.
The energy stability of the fully discrete schemes were only shown under the assumption
that the stabilizer depends implicitly on the uniform bound of the numerical solution.

The numerical method we will adopt in this work is the so-called exponential time differencing (ETD),
which involves exact integration of the governing equations
followed by an explicit approximation of a temporal integral involving the nonlinear terms.
The  ETD schemes was systematically studied  in \cite{BeKeVo98}
and then further developed by Cox and Matthews with the applications on stiff systems \cite{CoMa02},
where higher-order multistep and Runge-Kutta versions of these schemes were described.
Hochbruck and Ostermann provided several nice reviews
on the ETD Runge-Kutta methods \cite{HoOs05} and the ETD multistep methods \cite{HoOs10}.
In addition, the convergence of these methods were analyzed in detail under the analytical framework therein.
The linear stabilities of some ETD and modified ETD schemes were investigated by Du and Zhu \cite{DuZh04,DuZh05}.

A distinctive feature of ETD schemes is the exact evaluation of the contribution of the linear part,
which provides satisfactory stability and accuracy even though the linear terms have strong stiffness.
Such an advantage leads to some successful applications of ETD schemes on phase field models
which usually yield highly stiff ODE systems after suitable spatial discretizations.
Ju et al. developed stable and compact ETD schemes and their fast implementations
for Allen-Cahn \cite{JuZhZhDu15,ZhJuZh16}, Cahn-Hilliard \cite{JuZhDu15},
and elastic bending energy models \cite{WaJuDu16} by utilizing suitable linear splitting techniques.
All the proposed ETD schemes are explicit and thus highly efficient for  practical implementations.
A localized compact ETD algorithm based on the overlapping domain decomposition
was firstly used in \cite{ZhZhWaJuDu16} for extreme-scale phase field simulations
of three-dimensional coarsening dynamics in the supercomputer,
and the results showed excellent parallel scalability of the method.
In \cite{JuLiQiZh18}, the ETD multistep method was applied on the epitaxial growth model without slope selection \cite{LiLi03},
and the energy stability and the error estimates were established rigorously,
which is the first work to analyze the energy stability
and convergence of the ETD schemes for phase field models in the theoretical level.
To complete the theoretical analysis,
there is no need for any assumptions on the numerical solutions
due to the specific property of the logarithm term in the no-slope-selection model.
However, for other phase field models, such as the Cahn-Hilliard equation,
the assumptions on the uniform boundedness of the numerical solutions
or the Lipschitz continuity of some nonlinear functions are inevitable to ensure the energy stability.
Therefore, for the models whose solutions satisfy the maximum principle essentially,
it is highly desired to develop numerical approximations preserving the maximum principle in the discrete sense.

One of the typical phase field models  satisfying the maximum principle is the LAC equation \eqref{localAC}.
Recently, there have been some investigations on the maximum principle preserving numerical schemes for \eqref{localAC}.
Tang and Yang \cite{TaYa16} proved that the first order implicit-explicit schemes, with or without the stabilizing term,
preserve the maximum principle under some condition on the time step size.
Then, the energy stability and the maximum-norm error estimates are obtained by using the discrete maximum principle.
Shen et al. \cite{ShTaYa16} generalized the results presented in \cite{TaYa16} to the case of
the Allen-Cahn-like equation in a more abstract form with the potential and mobility satisfying some certain conditions.
Hou et al. \cite{HoTaYa17} studied the numerical approximation of the fractional Allen-Cahn equation
by considering the conventional Crank-Nicolson scheme.
They proved that the Crank-Nicolson scheme preserves the maximum principle
and this is the first work on the second order schemes preserving the maximum principle.
More than ten years ago, Du and Zhu \cite{DuZh05} showed that
the first order ETD scheme in the space-continuous version for \eqref{localAC} satisfies the maximum principle,
where some properties of the heat kernel were used in their proof.
However, the fully discrete ETD schemes were never studied.

The organization of this paper is as follows.
In Section 2, we construct the first and second order ETD time-stepping schemes for the NAC equation
with the quadrature-based finite difference approximation being used for spatial discretization.
Efficient implementation issues of the schemes are also briefly discussed.
In Section 3, both schemes are shown to satisfy the discrete maximum principle unconditionally.
Error estimates and asymptotic compatibility of the schemes are obtained in Section 4
and the discrete energy stability proved in Section 5.
Various numerical experiments are carried out in Section 6 to verify the theoretical results
and to investigate the effects of the nonlocal parameters.
Finally, some concluding remarks are given in Section 7.

\section{Fully discrete exponential time differencing schemes}

In this section, we present the  fully discrete ETD schemes for the NAC equation in general dimensions,
where the finite difference method, based on the contribution made in \cite{DuTaTiYa18},
is adopted for the spatial discretization of the nonlocal diffusion operator.
In particular, we also give the specific expression of the discrete nonlocal operator in 2D later.

\subsection{Quadrature-based finite difference semi-discretization}

Given a positive integer $N$,
we set $h=X/N$ as the uniform square mesh size and define $\bx_\bi=h\bi$ as the nodes in the mesh,
where $\bi\in\Z^d$ denotes a multi-index.
Let $\Omega_h$ be the set of nodes in the domain $\overline{\Omega}$.
At any node $\bx_\bi$, the nonlocal operator \eqref{nonlocalopt_def} can be rewritten as
\begin{equation}
\label{nonlocalopt_node}
\hL_\delta u(\bx_\bi)=
\frac{1}{2}\int_{B_\delta(\bo)}\frac{u(\bx_\bi+\bs)+u(\bx_\bi-\bs)-2u(\bx_\bi)}{|\bs|^2}\|\bs\|_1
\cdot\frac{|\bs|^2}{\|\bs\|_1}\rho_\delta(|\bs|)\,\d\bs,
\end{equation}
where $\|\cdot\|_1$  stands for the vector $1$-norm.
Then, a quadrature-based finite difference discretization of the nonlocal operator \eqref{nonlocalopt_node}
can be defined as \cite{DuTaTiYa18}
$$\hL_{\delta,h}u(\bx_\bi)
=\frac{1}{2}\int_{B_\delta(\bo)}\hI_h\bigg(\frac{u(\bx_\bi+\bs)+u(\bx_\bi-\bs)-2u(\bx_\bi)}{|\bs|^2}\|\bs\|_1\bigg)
\frac{|\bs|^2}{\|\bs\|_1}\rho_\delta(|\bs|)\,\d\bs,$$
where $\hI_h$ represents the piecewise $d$-multilinear interpolation operator with respect to $\bs$ associated with the mesh.
More precisely, for a function $v(\bs)$,
the interpolation $\hI_hv(\bs)$ is piecewise linear with respect to each component of the spatial variable $\bs$ and
$$\hI_hv(\bs)=\sum_{\bs_\bj}v(\bs_\bj)\psi_\bj(\bs),$$
where $\psi_\bj$ is the piecewise $d$-multilinear basis function
satisfying $\psi_\bj(\bs_\bi)=0$ when $\bi\not=\bj$ and $\psi_\bj(\bs_\bj)=1$.
Therefore, the resulting quadrature-based finite difference discretization
of the nonlocal operator \eqref{nonlocalopt_node} reads
\begin{equation}
\label{nonlocalopt_discrete}
\hL_{\delta,h}u(\bx_\bi)
=\sum_{\bo\not=\bs_\bj\in B_\delta(\bo)}
\frac{u(\bx_\bi+\bs_\bj)+u(\bx_\bi-\bs_\bj)-2u(\bx_\bi)}{|\bs_\bj|^2}\|\bs_\bj\|_1\beta_\delta(\bs_\bj),
\quad\bx_\bi\in\Omega_h,
\end{equation}
where the periodicity conditions are used for the nodes not in $\Omega_h$, and
\begin{equation}
\label{nonlocalopt_discoef_beta}
\beta_\delta(\bs_\bj)=\frac{1}{2}\int_{B_\delta(\bo)}\psi_\bj(\bs)\frac{|\bs|^2}{\|\bs\|_1}\rho_\delta(|\bs|)\,\d\bs.
\end{equation}
It is easy to check that the operator $\hL_{\delta,h}$ is self-adjoint and negative semi-definite.

The discretized scheme \eqref{nonlocalopt_discrete} is proposed in \cite{DuTaTiYa18}
for the problem with a homogeneous Dirichlet-type nonlocal constraint
and it has been proved \cite{DuTaTiYa18,TiJuDu17} that, for any fixed $\delta>0$,
the discrete operators $\hL_{\delta,h}$ is consistent to $\hL_\delta$
with the errors $\mathcal{O}(h^2)$ as $h\to0$.
For the case of periodic boundary condition considered here, all similar estimates also hold,
so we give the following consistency estimates without proof.

\begin{lemma}
\label{lem_discrete_consistency}
Assume that $u\in\Cp^4(\overline{\Omega})$,
then it holds that
\begin{equation}
\label{discrete_consistency}
\max_{\bx_\bi\in\Omega_h}|\hL_{\delta,h}u(\bx_\bi)-\hL_\delta u(\bx_\bi)|\le Ch^2\|u\|_{C^4},
\end{equation}
where $C>0$ is a constant independent of $\delta$ and $h$.
\end{lemma}

By ordering the nodes in the lexicographical order,
we can obtain the nonlocal stiffness matrix, denoted by $D_h\in\R^{dN\times dN}$, associated with $\hL_{\delta,h}$.
It is obvious that $D_h$ is symmetric, negative semi-definite,
and weakly diagonally dominant with all negative diagonal entries.
The space-discrete scheme of \eqref{nonlocalAC} is to find a vector-valued function $U:[0,T]\to\R^{dN}$ such that
\begin{subequations}
\label{nonlocalAC_diff}
\begin{align}
\daoshu{U}{t} & =\eps^2D_hU+U-U^{.3},\quad t\in(0,T],\label{nonlocalAC_diff1}\\
U(0) & =U_0,
\end{align}
\end{subequations}
where $U^{.3}=(U_1^3,U_2^3,\dots,U_{dN}^3)^T$ and
$U_0\in\R^{dN}$ is given by the initial data.
For the sake of the stability of the time-stepping schemes developed later,
we introduce a {\em stabilizing parameter} $\kappa>0$ and define
\begin{equation}
\label{def_operators}
L_h:=-\eps^2D_h+\kappa I_{dN},\qquad f(U):=(\kappa+1)U-U^{.3},
\end{equation}
where $I_{dN}$ is the $dN\times dN$ identity matrix,
so $L_h$ is symmetric, positive definite,
and strictly diagonally dominant with all positive diagonal entries.
Then, the ODE system \eqref{nonlocalAC_diff1} could be written as
$$\daoshu{U}{t}+L_hU=f(U),$$
whose solution satisfies
\begin{equation}
\label{nonlocalAC_diff_solution}
U(t+\tau)=\e^{-L_h\tau}U(t)+\int_0^\tau\e^{-L_h(\tau-s)}f(U(t+s))\,\d s,\quad\forall\,t\ge0,\ \tau>0.
\end{equation}
In the above we have used a property of the differentiation of matrix exponentials (see Lemma \ref{lem_matfun} (5)),
and we list below some other properties of matrix functions (see \cite{Matfun08}) useful to the analysis later.

\begin{lemma}[see \cite{Matfun08}]
\label{lem_matfun}
Let $\phi$ be defined on the spectrum of $A\in\C^{m\times m}$, that is, the values
$$\phi^{(j)}(\lambda_i),\qquad 0\le j\le n_i-1,\ 1\le i\le m$$
exist, where $\{\lambda_i\}_{i=1}^m$ are the eigenvalues of $A$,
and $n_i$ is the order of the largest Jordan block where $\lambda_i$ appears.
Then

{\rm(1)} $\phi(A)$ commutes with $A$;

{\rm(2)} $\phi(A^T)=\phi(A)^T$;

{\rm(3)} the eigenvalues of $\phi(A)$ are $\{\phi(\lambda_i):1\le i\le m\}$;

{\rm(4)} $\phi(P^{-1}AP)=P^{-1}\phi(A)P$ for any nonsingular matrix $P\in\C^{m\times m}$;

{\rm(5)} $\daoshu{}{s}(\e^{As})=A\e^{As}=\e^{As}A$ for any $s\in\R$.
\end{lemma}

\subsection{Exponential time differencing schemes for time-stepping}

Given a positive integer $K_t$, we divide the time interval by $\{t_n=n\tau:0\le n\le K_t\}$ with a uniform time step $\tau=T/K_t$.
Setting $t=t_n$ in \eqref{nonlocalAC_diff_solution} gives us
\begin{equation}
\label{nonlocalAC_diff_solution2}
U(t_{n+1})=\e^{-L_h\tau}U(t_n)+\int_0^\tau\e^{-L_h(\tau-s)}f(U(t_n+s))\,\d s.
\end{equation}
The first order ETD (ETD1) scheme comes from
approximating $f(U(t_n+s))$ by $f(U(t_n))$ in $s\in[0,\tau]$ and calculating the produced integral exactly \cite{CoMa02}.
The ETD1 scheme of \eqref{nonlocalAC} reads: for $n=0,1,\cdots,K_t-1$,
\begin{equation}
\label{nonlocalAC_ETD1}
U^{n+1}=\e^{-L_h\tau}U^n+\int_0^\tau\e^{-L_h(\tau-s)}f(U^n)\,\d s,
\end{equation}
that is,
\begin{equation}
\label{nonlocalAC_ETD1var}
U^{n+1}=\phi_0(L_h\tau)U^n+\tau\phi_1(L_h\tau)f(U^n),
\end{equation}
where
$$\phi_0(a):=\e^{-a},\quad\phi_1(a):=\frac{1-\e^{-a}}{a},\quad a\not=0.$$
The second order ETD Runge-Kutta (ETDRK2) scheme is obtained by
approximating $f(U(t_n+s))$ by a linear interpolation based on $f(U(t_n))$ and $f(\widetilde{U}^{n+1})$,
where $\widetilde{U}^{n+1}$ is an approximation of $U(t_{n+1})$.
The ETDRK2 scheme of \eqref{nonlocalAC} takes the form: for $n=0,1,\cdots,K_t-1$,
\begin{subequations}
\label{nonlocalAC_ETDRK2}
\begin{align}
\widetilde{U}^{n+1} & =\e^{-L_h\tau}U^n+\int_0^\tau\e^{-L_h(\tau-s)}f(U^n)\,\d s,\label{nonlocalAC_ETDRK2eq1}\\
U^{n+1} & =\e^{-L_h\tau}U^n+\int_0^\tau\e^{-L_h(\tau-s)}\Big[\Big(1-\frac{s}{\tau}\Big)f(U^n)+\frac{s}{\tau}f(\widetilde{U}^{n+1})\Big]\,\d s,
\label{nonlocalAC_ETDRK2eq2}
\end{align}
\end{subequations}
or equivalently,
\begin{subequations}
\label{nonlocalAC_ETDRK2var}
\begin{align}
\widetilde{U}^{n+1} & =\phi_0(L_h\tau)U^n+\tau\phi_1(L_h\tau)f(U^n),\label{nonlocalAC_ETDRK2var1}\\
U^{n+1} & =\widetilde{U}^{n+1}+\tau\phi_2(L_h\tau)(f(\widetilde{U}^{n+1})-f(U^n)),\label{nonlocalAC_ETDRK2var2}
\end{align}
\end{subequations}
where
$$\phi_2(a):=\frac{\e^{-a}-1+a}{a^2},\quad a\not=0.$$
We know that $\phi_0(a)$, $\phi_1(a)$, and $\phi_2(a)$ are all positive when $a>0$.

\subsection{Efficient implementations of the ETD schemes}

We close this section by giving a brief illustration on the practical implementation
of the proposed schemes \eqref{nonlocalAC_ETD1var} and \eqref{nonlocalAC_ETDRK2var}.
Using the 2D case as the example,
we first give the explicit formula of the discrete operator $\hL_{\delta,h}$ (as illustrated in \cite{DuTaTiYa18})
and then discuss efficient implementation of the actions of the matrix exponentials.


Let $u_{i,j}$ be the nodal value of the numerical solution at the mesh point $(x_i,y_j)\in\Omega_h\subset\R^2$
 and $r=[\delta/h]+1$ be the smallest integer larger than $\delta/h$.
Then, we have
\begin{equation}
\hL_{\delta,h}u_{i,j}=\sum_{p=0}^r\sum_{q=0}^rc_{p,q}(u_{i+p,j+q}+u_{i-p,j+q}+u_{i+p,j-q}+u_{i-p,j-q}-4u_{i,j}),
\end{equation}
where $c_{0,0}=0$ and
\begin{equation}
\label{coef_cpq}
c_{p,q}=\frac{p+q}{(p^2+q^2)h}\iint_{B_\delta^+}\psi_{p,q}(x,y)\rho_\delta(\sqrt{x^2+y^2})\frac{x^2+y^2}{x+y}\,\d x\d y,
\end{equation}
with $\psi_{p,q}$ denoting the bilinear basis function located at the point $(ph,qh)$
and $B_\delta^+$ the first quadrant of the disc centered at the origin with radius $\delta$.
Note that $c_{p,q}=c_{q,p}$ for any $p$ and $q$.
One can apply efficient quadrature rules on the double integrals in \eqref{coef_cpq}.


We represent $U^n\in\R^{N\times N}$ in the matrix form with entries $u_{i,j}^n$
and define the operator $\hL_h=\kappa\hI-\eps^2\hL_{\delta,h}$ whose matrix form is given by $L_h$ defined in \eqref{def_operators},
where $\hI$ is the identity mapping.
The key process of calculating $U^{n+1}$ from the scheme \eqref{nonlocalAC_ETD1var} or \eqref{nonlocalAC_ETDRK2var}
is the efficient implementation of the actions of the operator exponentials $\phi_\gamma(\hL_h\tau)$, $\gamma=0,1,2$.
Since $\hL_h$ comes from the discretization of $\hL_\delta$ with the periodic boundary condition,
the exponentials $\phi_\gamma(\hL_h\tau)$ can be implemented by the 2D discrete Fourier transform (DFT).
More precisely, if we denote by $\mathcal{F}$ the 2D DFT operator,
then, for any $V=(V_{k,l})\in\C^{N\times N}$,
the action of the operator $\widehat{\hL}_h:=\mathcal{F}\hL_h\mathcal{F}^{-1}$ can be implemented via
$$(\widehat{\hL}_hV)_{k,l}=\lambda_{k,l}V_{k,l},\quad 1\le k,l\le N,$$
where $\lambda_{k,l}$'s, the eigenvalues of $\hL_h$, are given by
\begin{equation*}
\lambda_{k,l}=\kappa+4\eps^2\sum_{p=0}^r\sum_{q=0}^rc_{p,q}\Big(1-\cos\frac{2\pi(k-1)p}{N}\cos\frac{2\pi(l-1)q}{N}\Big),
\quad 1\le k,l\le N.
\end{equation*}
According to Lemma \ref{lem_matfun} (4), we have
$$\phi_\gamma(\hL_h\tau)=\mathcal{F}^{-1}\phi_\gamma(\widehat{\hL}_h\tau)\mathcal{F},
\quad(\phi_\gamma(\widehat{\hL}_h\tau)V)_{k,l}=\phi_\gamma(\lambda_{k,l}\tau)V_{k,l},\quad\gamma=0,1,2.$$
The actions of $\mathcal{F}$ and $\mathcal{F}^{-1}$ can be implemented
by the 2D fast Fourier transform (FFT) and its inverse transform, respectively.
Such implementation can be naturally generalized to higher-dimensional
spaces and the computational complexity is thus $\mathcal{O}(N^d\log N)$ per time step.

%

\section{Discrete maximum principle}

Denote by $\|\cdot\|_\infty$ the standard vector or matrix $\infty$-norm,
and by $\|\cdot\|_2$ the standard vector or matrix $2$-norm.
The following lemma is a special case of Theorem 2 in \cite{Lazer71}.

\begin{lemma}
\label{lem_nddsystem}
Let $A=(a_{ij})\in\R^{m\times m}$ with $a_{ii}<0$, $1\le i\le m$,
and there exists $\kappa>0$ such that
$$|a_{ii}|\ge\sum_{j=1}^m|a_{ij}|+\kappa,\quad1\le i\le m,$$
then the nontrivial solution $\theta=\theta(t)$ to the linear differential system
\begin{equation}
\label{nddsystem}
\daoshu{\theta}{t}=A\theta,\quad t>0
\end{equation}
satisfies
$$\|\theta(t_2)\|_\infty\le\e^{-\kappa(t_2-t_1)}\|\theta(t_1)\|_\infty,\quad\forall\,t_2\ge t_1\ge0.$$
\end{lemma}

The following result is a key ingredient to prove the discrete maximum principle.

\begin{lemma}
\label{lem_invnorm}
For any $\kappa>0$ and $\tau>0$, we always have $\|\e^{-L_h\tau}\|_\infty\le\e^{-\kappa\tau}$.
\end{lemma}

\begin{proof}
Since $L_h$ is strictly diagonally dominant with all positive diagonal entries,
the matrix $A:=-L_h$ satisfies the conditions of Lemma \ref{lem_nddsystem} with $\kappa>0$.
For any nonzero $\theta_0\in\R^{dN}$,
we know that the solution to the linear differential system \eqref{nddsystem} with the initial value $\theta(0)=\theta_0$
is given by $\theta(t)=\e^{-L_ht}\theta_0$ and,
by using Lemma \ref{lem_nddsystem}, satisfies
$$\|\e^{-L_h\tau}\theta_0\|_\infty=\|\theta(\tau)\|_\infty\le\e^{-\kappa\tau}\|\theta_0\|_\infty,\quad\tau>0.$$
Therefore, the result follows from the arbitrariness of $\theta_0$.
\end{proof}

\begin{remark}
Although our deduction above is restricted to the case of periodic boundary condition,
the result of Lemma \ref{lem_invnorm} is also suitable for the case of the Dirichlet boundary condition,
since the corresponding nonlocal stiffness matrix $D_h$ is still weakly diagonally dominant with all negative diagonal entries.
The analysis results in this paper  could be obtained similarly for the Dirichlet boundary condition.
\end{remark}

Since the nonlinear mapping $f:\R^{dN}\to\R^{dN}$ defined in \eqref{def_operators}
is actually a set of $dN$ independent one-variable functions,
we just need to consider anyone of them.

\begin{lemma}
\label{lem_rhsnorm}
Define $f_0(\xi)=(\kappa+1)\xi-\xi^3$ for any $\xi\in\R$.
If $\kappa\ge 2$, then
$$|f_0(\xi)|\le\kappa,\quad\forall\,\xi\in[-1,1].$$
\end{lemma}

\begin{proof}
Obviously, $f_0(-1)=-\kappa$ and $f_0(1)=-\kappa$.
For any $\xi\in[-1,1]$, if $\kappa\ge2$, we have
\begin{equation}
\label{lem_rhsnorm_f0p}
f_0'(\xi)=\kappa+1-3\xi^2\ge \kappa-2\ge0,
\end{equation}
which gives us the result.
\end{proof}

\begin{theorem}
\label{thm_discrete_max_ETD1}
Assume that the initial data satisfies $\|u_0\|_{L^\infty}\le1$.
Then, for any time step size $\tau>0$, the ETD1 scheme \eqref{nonlocalAC_ETD1} preserves the discrete maximum principle, i.e.,
$$\|U^n\|_\infty\le1,\quad\forall\,n\ge0,$$
provided the stabilizing parameter $\kappa\ge 2$.
\end{theorem}

\begin{proof}
We prove this theorem by induction.
Obviously, it holds $\|U^0\|_\infty\le\|u_0\|_{L^\infty}\le1$.
Now assume that the result holds for $n=k$, i.e., $\|U^k\|_\infty\le1$.
Next we check this holds for $n=k+1$.
According to the scheme \eqref{nonlocalAC_ETD1}, we have
\begin{equation*}
\|U^{k+1}\|_\infty\le\|\e^{-L_h\tau}\|_\infty\|U^k\|_\infty+\int_0^\tau\|\e^{-L_h(\tau-s)}\|_\infty\,\d s\cdot\|f(U^k)\|_\infty.
\end{equation*}
It follows from Lemma \ref{lem_invnorm} that
\begin{equation}
\label{thm_discrete_max_ETD1pf}
\|\e^{-L_h\tau}\|_\infty\le\e^{-\kappa\tau},\qquad
\int_0^\tau\|\e^{-L_h(\tau-s)}\|_\infty\,\d s\le\int_0^\tau\e^{-\kappa(\tau-s)}\,\d s=\frac{1-\e^{-\kappa\tau}}{\kappa}.
\end{equation}
Using Lemma \ref{lem_rhsnorm} and $\|U^k\|_\infty\le1$, we have $\|f(U^k)\|_\infty\le\kappa$.
Consequently,
$$\|U^{k+1}\|_\infty\le\e^{-\kappa\tau}\cdot1+\frac{1-\e^{-\kappa\tau}}{\kappa}\cdot \kappa=1,$$
which completes the proof.
\end{proof}

\begin{theorem}
\label{thm_discrete_max_ETDRK2}
Assume that the initial data satisfies $\|u_0\|_{L^\infty}\le1$.
Then, for any time step size $\tau>0$, the ETDRK2 scheme \eqref{nonlocalAC_ETDRK2} preserves the discrete maximum principle, i.e.,
$$\|U^n\|_\infty\le1,\quad\forall\,n\ge0,$$
provided the stabilizing parameter $\kappa\ge 2$.
\end{theorem}

\begin{proof}
We again prove this by induction.
Obviously, it holds $\|U^0\|_\infty\le\|u_0\|_{L^\infty}\le1$.
Now assume that the result holds for $n=k$, i.e., $\|U^k\|_\infty\le1$.
Next we check this for $n=k+1$.
According to the formula \eqref{nonlocalAC_ETDRK2eq1} and the proof of Theorem \ref{thm_discrete_max_ETD1},
we have $\|\widetilde{U}^{k+1}\|_\infty\le1$.
According to the formula \eqref{nonlocalAC_ETDRK2eq2}, we have
\begin{equation}
\label{thm_discrete_max_ETDRK2pf}
\|U^{k+1}\|_\infty\le\|\e^{-L_h\tau}\|_\infty\|U^k\|_\infty
+\int_0^\tau\|\e^{-L_h(\tau-s)}\|_\infty\Big\|\Big(1-\frac{s}{\tau}\Big)f(U^k)+\frac{s}{\tau}f(\widetilde{U}^{k+1})\Big\|_\infty\,\d s.
\end{equation}
Since $\|U^k\|_\infty\le1$ and $\|\widetilde{U}^{k+1}\|_\infty\le1$, using Lemma \ref{lem_rhsnorm}, we have
$$\|f(U^k)\|_\infty\le \kappa,\quad\|f(\widetilde{U}^{k+1})\|_\infty\le \kappa,$$
and then, for $s\in[0,\tau]$,
$$\Big\|\Big(1-\frac{s}{\tau}\Big)f(U^k)+\frac{s}{\tau}f(\widetilde{U}^{k+1})\Big\|_\infty
\le\Big(1-\frac{s}{\tau}\Big)\kappa+\frac{s}{\tau}\kappa=\kappa.$$
Again, by using \eqref{thm_discrete_max_ETD1pf}, we obtain from \eqref{thm_discrete_max_ETDRK2pf} that
$$\|U^{k+1}\|_\infty\le\e^{-\kappa\tau}\cdot1+\frac{1-\e^{-S\tau}}{\kappa}\cdot \kappa=1,$$
which completes the proof.
\end{proof}

\begin{remark}
We will then always require $\kappa\geq 2$ for the proposed ETD1 and ETDRK2
schemes in the rest of the paper
so that they preserve the discrete maximum principle.
\end{remark}

\section{Error estimates and asymptotic compatibility}

We will analyze the two types of convergence behaviors of the numerical solution to
the ETD1 scheme \eqref{nonlocalAC_ETD1} and the ETDRK2 scheme \eqref{nonlocalAC_ETDRK2}, respectively.
First, for any fixed $\delta>0$, we prove that
the numerical solution converges to the exact solution of the NAC equation \eqref{nonlocalAC}
as the spatial mesh size $h$ and the time step size $\tau$ go to zero.
Second, we show that the numerical solution converges to the exact solution of the LAC equation \eqref{localAC}
as the horizon parameter $\delta$, the spatial size $h$, and the temporal step $\tau$ approach to zero.
The latter convergence behavior of the numerical solution is often called the {\em asymptotic compatibility} \cite{TiDu14}
in nonlocal modeling.

We first establish the $L^\infty$ error estimates for the numerical solution produced by the ETD1 scheme \eqref{nonlocalAC_ETD1}
for the NAC equation \eqref{nonlocalAC} with any fixed $\delta>0$.

\begin{theorem}
\label{thm_error_ETD1}
Given a fixed $\delta>0$.
Assume that the exact solution $u$ to the NAC equation \eqref{nonlocalAC}
belongs to $C^1([0,T];\Cp^4(\overline{\Omega}))$
and $\{U^n\}_{n=0}^{K_t}$ is generated by the ETD1 scheme \eqref{nonlocalAC_ETD1} with $U^0=I^hu_0$.
If $\|u_0\|_{L^\infty}\le1$, then we have
\begin{equation}
\label{erretd1}
\|U^n-I^hu(t_n)\|_\infty\le C\e^{t_n}(h^2+\tau),\quad t_n\le T
\end{equation}
for any $h>0$ and $\tau>0$,
where the constant $C>0$ depends on the $C^1([0,T];\Cp^4(\overline{\Omega}))$ norm of $u$,
but independent of $\delta$, $h$ and $\tau$.
\end{theorem}

\begin{proof}
Recalling the construction of the ETD1 scheme \eqref{nonlocalAC_ETD1},
we observe that, for a known $U^n$, the solution $U^{n+1}$ is actually given by $U^{n+1}=W_1(\tau)$
with the function $W_1:[0,\tau]\to\R^{dN}$ determined by the following evolution equation
\begin{equation}
\label{nonlocalAC_ETD1_heat}
\begin{dcases}
\daoshu{W_1(s)}{s}=-\kappa W_1(s)+\eps^2D_hW_1(s)+f(U^n), & s\in(0,\tau),\\
W_1(0)=U^n.
\end{dcases}
\end{equation}
Then, for the NAC equation \eqref{nonlocalAC},
we can give a similar illustration as follows:
for given $u(\bx,t_n)$, the solution $u(\bx,t_{n+1})$ is determined by $u(\bx,t_{n+1})=w(\bx,\tau)$
with the function $w(\bx,s)$ satisfying
\begin{equation}
\label{nonlocalAC_heat}
\begin{dcases}
\piandao{w}{s}=-\kappa w+\eps^2\hL_\delta w+f(w), & \bx\in\Omega,\ s\in(0,\tau),\\
\text{$w(\cdot,s)$ is $\Omega$-periodic}, & s\in[0,\tau],\\
w(\bx,0)=u(\bx,t_n), & \bx\in\overline{\Omega}.
\end{dcases}
\end{equation}
Let $e_1(s)=W_1(s)-I^hw(s)$, where $I^h$ is the operator limiting a function on the mesh $\Omega_h$.
Then, the difference between \eqref{nonlocalAC_ETD1_heat} and \eqref{nonlocalAC_heat} yields
\begin{equation}
\label{nonlocalAC_ETD1_errorheat}
\begin{dcases}
\daoshu{e_1(s)}{s}=-L_he_1(s)+f(U^n)-f(I^hu(t_n))+R_{h\tau}^{(1)}(s), & s\in(0,\tau),\\
e_1(0)=U^n-I^hu(t_n)=:e_1^n,
\end{dcases}
\end{equation}
where $R_{h\tau}^{(1)}(s)$ is the truncated error, that is,
$$R_{h\tau}^{(1)}(s)=\eps^2(D_hI^hu(t_n+s)-I^h\hL_\delta u(t_n+s))+f(I^hu(t_n))-f(I^hu(t_n+s)).$$
Since $|f'(\xi)|\le \kappa+1$ for any $\xi\in[-1,1]$ and
both the exact and numerical solutions satisfy the maximum principles if $\kappa\ge 2$,
we have
\begin{equation}
\label{thm_error_ETD1pf1}
\|f(U^n)-f(I^hu(t_n))\|_\infty\le(\kappa+1)\|U^n-I^hu(t_n)\|_\infty=(\kappa+1)\|e_1^n\|_\infty.
\end{equation}
According to the consistency result given by Lemma \ref{lem_discrete_consistency},
if the exact solution $u$ is sufficiently smooth, at least $u\in C^1([0,T];\Cp^4(\overline{\Omega}))$,
then we have
\begin{equation}
\label{thm_error_ETD1pf2}
\|D_hI^hu(t)-I^h\hL_\delta u(t)\|_\infty\le C_1h^2,\quad\forall\,t\in(0,T],
\end{equation}
and for $s\in[0,\tau]$,
\begin{equation}
\|f(I^hu(t_n))-f(I^hu(t_n+s))\|_\infty\le(\kappa+1)\|I^h(u(t_n)-u(t_n+s))\|_\infty\le C_2(\kappa+1)\tau,
\end{equation}
where $C_1$ and $C_2$ depend on the $C^1([0,T];\Cp^4(\overline{\Omega}))$ norm of $u$, but independent of $\delta$, $h$ and $\tau$.
Thus, we obtain
\begin{equation}
\|R_{h\tau}^{(1)}(s)\|_\infty\le C(h^2+\tau),\quad\forall\,s\in[0,\tau],
\end{equation}
where $C=\max\{C_1\eps^2,C_2(S+1)\}$.
Integrating the ODE in \eqref{nonlocalAC_ETD1_errorheat} leads to
\begin{equation}
e_1(t)=\e^{-L_ht}e_1(0)+\int_0^t\e^{-L_h(t-s)}[f(U^n)-f(I^hu(t_n))+R_{h\tau}^{(1)}(s)]\,\d s,\quad t\in[0,\tau].
\end{equation}
Setting $t=\tau$ and using \eqref{thm_discrete_max_ETD1pf}, we have
\begin{align}
\|e_1^{n+1}\|_\infty
& \le\|\e^{-L_h\tau}\|_\infty\|e_1^n\|_\infty
+[(\kappa+1)\|e_1^n\|_\infty+C(h^2+\tau)]\int_0^\tau\|\e^{-L_h(\tau-s)}\|_\infty\,\d s\nonumber\\
& \le\e^{-\kappa\tau}\|e_1^n\|_\infty+\frac{1-\e^{-\kappa\tau}}{\kappa}[(\kappa+1)\|e_1^n\|_\infty+C(h^2+\tau)]\nonumber\\
& =\Big(1+\frac{1-\e^{-\kappa\tau}}{\kappa}\Big)\|e_1^n\|_\infty+\frac{1-\e^{-\kappa\tau}}{\kappa}C(h^2+\tau)\nonumber\\
& \le(1+\tau)\|e_1^n\|_\infty+C\tau(h^2+\tau),\label{thm_error_ETD1pf3}
\end{align}
where in the last step we have used the fact that $1-\e^{-s}\le s$ for any $s>0$.
Using the Gronwall's inequality, we obtain
\begin{align}
\|e_1^n\|_\infty
& \le(1+\tau)^n\|e_1^0\|_\infty+C\tau(h^2+\tau)\sum_{k=0}^{n-1}(1+\tau)^k\nonumber\\
& =(1+\tau)^n\|e_1^0\|_\infty+C[(1+\tau)^n-1](h^2+\tau)\nonumber\\
& \le\e^{n\tau}\|e_1^0\|_\infty+C\e^{n\tau}(h^2+\tau).
\end{align}
Therefore, we obtain \eqref{erretd1} since $e_1^0=0$ and $n\tau=t_n$.
\end{proof}

Now, we turn to the $L^\infty$ error estimates for the ETDRK2 scheme \eqref{nonlocalAC_ETDRK2}
with any fixed $\delta>0$.

\begin{theorem}
\label{thm_error_ETDRK2}
Given a fixed $\delta>0$.
Assume that the exact solution $u$ of the NAC equation \eqref{nonlocalAC}
belongs to $C^2([0,T];\Cp^4(\overline{\Omega}))$
and $\{U^n\}_{n=0}^{K_t}$ is generated by the ETDRK2 scheme \eqref{nonlocalAC_ETDRK2} with $U^0=I^hu_0$.
If $\|u_0\|_{L^\infty}\le1$, then we have
\begin{equation}
\label{erretd2}
\|U^n-I^hu(t_n)\|_\infty\le C\e^{t_n}(h^2+\tau^2),\quad t_n\le T
\end{equation}
for any $h>0$ and $1\geq \tau>0$,
where the constant $C>0$ depends on the $C^2([0,T];\Cp^4(\overline{\Omega}))$ norm of $u$,
but independent of $\delta$, $h$ and $\tau$.
\end{theorem}

\begin{proof}
For a known $U^n$,
the solution $U^{n+1}$ to the ETDRK2 scheme \eqref{nonlocalAC_ETDRK2} is actually given by $U^{n+1}=W_2(\tau)$
with the function $W_2:[0,\tau]\to\R^{dN}$ determined by the evolution equation
\begin{equation}
\label{nonlocalAC_ETDRK2_heat}
\begin{dcases}
\daoshu{W_2(s)}{s}=-\kappa W_2(s)+\eps^2D_hW_2(s)
+\Big(1-\frac{s}{\tau}\Big)f(U^n)+\frac{s}{\tau}f(\widetilde{U}^{n+1}), & s\in(0,\tau),\\
W_2(0)=U^n,
\end{dcases}
\end{equation}
where $\widetilde{U}^{n+1}$, defined by \eqref{nonlocalAC_ETDRK2eq1},
is given by $\widetilde{U}^{n+1}=W_1(\tau)$ with $W_1(s)$ satisfying \eqref{nonlocalAC_ETD1_heat}.
Let $e_2(s)=W_2(s)-I^hw(s)$.
The difference between \eqref{nonlocalAC_ETDRK2_heat} and \eqref{nonlocalAC_heat} leads to
\begin{equation}
\label{nonlocalAC_ETDRK2_errorheat}
\begin{dcases}
\daoshu{e_2(s)}{s}=-L_he_2(s)
+\Big(1-\frac{s}{\tau}\Big)[f(U^n)-f(I^hu(t_n))]\\
\qquad\qquad\quad +\frac{s}{\tau}[f(\widetilde{U}^{n+1})-f(I^hu(t_{n+1}))]+R_{h\tau}^{(2)}(s), & s\in(0,\tau),\\
e_2(0)=U^n-I^hu(t_n)=:e_2^n,
\end{dcases}
\end{equation}
where $R_{h\tau}^{(2)}(s)$ is the truncated error given by
\begin{align*}
R_{h\tau}^{(2)}(s) & =\eps^2(D_hI^hu(t_n+s)-I^h\hL_\delta u(t_n+s))\\
& \qquad +\Big[\Big(1-\frac{s}{\tau}\Big)f(I^hu(t_n))+\frac{s}{\tau}f(I^hu(t_{n+1}))-f(I^hu(t_n+s))\Big].
\end{align*}
According to error estimates for the linear interpolation,
we have, for $s\in[0,\tau]$, that
$$\Big\|\Big(1-\frac{s}{\tau}\Big)f(I^hu(t_n))+\frac{s}{\tau}f(I^hu(t_{n+1}))-f(I^hu(t_n+s))\Big\|\le C_3\tau^2,$$
where $C_3$ depends on the $C^2([0,T];\Cp^4(\overline{\Omega}))$ norm of $u$, but independent of $\delta$, $h$ and $\tau$.
Thus, combining with \eqref{thm_error_ETD1pf2}, we obtain
\begin{equation}
\|R_{h\tau}^{(2)}(s)\|_\infty\le C_4(h^2+\tau^2),\quad\forall\,s\in[0,\tau],
\end{equation}
where $C_4=\max\{C_1\eps^2,C_3\}$.
According to \eqref{thm_error_ETD1pf3} in the proof for the ETD1 scheme, we have
$$\|\widetilde{U}^{n+1}-I^hu(t_{n+1})\|_\infty
\le(1+\tau)\|U^n-I^hu(t_n)\|_\infty+C_5\tau(h^2+\tau),$$
where $C_5$ depends on the $C^1([0,T];\Cp^4(\overline{\Omega}))$ norm of $u$. Then, using the Lipschitz continuity of $f$, we obtain
\begin{align*}
\|f(\widetilde{U}^{n+1})-f(I^hu(t_{n+1}))\|_\infty
& \le(\kappa+1)\|\widetilde{U}^{n+1}-I^hu(t_{n+1})\|_\infty\\
& \le(\kappa+1)(1+\tau)\|e_2^n\|_\infty+C_5\tau(\kappa+1)(h^2+\tau).
\end{align*}
Combining with \eqref{thm_error_ETD1pf1}, we have, for any $s\in(0,\tau)$, that
\begin{align}
& \Big\|\Big(1-\frac{s}{\tau}\Big)[f(U^n)-f(I^hu(t_n))]+\frac{s}{\tau}[f(\widetilde{U}^{n+1})-f(I^hu(t_{n+1}))]\Big\|_\infty\nonumber\\
& \qquad
\le\Big(1-\frac{s}{\tau}\Big)(\kappa+1)\|e_2^n\|_\infty+\frac{s}{\tau}(\kappa+1)(1+\tau)\|e_2^n\|_\infty+C_5s(\kappa+1)(h^2+\tau)\nonumber\\
& \qquad =(s+1)(\kappa+1)\|e_2^n\|_\infty+C_5s(\kappa+1)(h^2+\tau).
\end{align}
Integrating the ODE in \eqref{nonlocalAC_ETDRK2_errorheat} leads to
\begin{align}
e_2(t) & =\e^{-L_ht}e_2(0)+\int_0^t\e^{-L_h(t-s)}\Big\{\Big(1-\frac{s}{\tau}\Big)[f(U^n)-f(I^hu(t_n))]\nonumber\\
& \qquad +\frac{s}{\tau}[f(\widetilde{U}^{n+1})-f(I^hu(t_{n+1}))]+R_{h\tau}^{(2)}(s)\Big\}\,\d s,\qquad t\in[0,\tau].
\end{align}
Setting $t=\tau$ and using \eqref{thm_discrete_max_ETD1pf}, we have
\begin{align*}
\|e_2^{n+1}\|_\infty
& \le\|\e^{-L_h\tau}\|_\infty\|e_2^n\|_\infty
+[(\kappa+1)\|e_2^n\|_\infty+C_4(h^2+\tau^2)]\int_0^\tau\|\e^{-L_h(\tau-s)}\|_\infty\,\d s\\
& \qquad +[(\kappa+1)\|e_2^n\|_\infty+C_5(\kappa+1)(h^2+\tau)]\int_0^\tau s\|\e^{-L_h(\tau-s)}\|_\infty\,\d s\\
& \le\e^{-\kappa\tau}\|e_2^n\|_\infty+\frac{1-\e^{-\kappa\tau}}{\kappa}[(\kappa+1)\|e_2^n\|_\infty+C_4(h^2+\tau^2)]\\
& \qquad +\frac{\e^{-\kappa\tau}-1+\kappa\tau}{\kappa^2}[(\kappa+1)\|e_2^n\|_\infty+C_5(\kappa+1)(h^2+\tau)]\\
& =\Big(1+\tau+\frac{\e^{-\kappa\tau}-1+\kappa\tau}{\kappa^2}\Big)\|e_2^n\|_\infty\\
& \qquad +\frac{1-\e^{-\kappa\tau}}{\kappa}C_4(h^2+\tau^2)+\frac{\e^{-\kappa\tau}-1+\kappa\tau}{\kappa^2}C_5(\kappa+1)(h^2+\tau)\\
& \le\Big(1+\tau+\frac{\tau^2}{2}\Big)\|e_2^n\|_\infty+C_4\tau(h^2+\tau^2)+\frac{1}{2}C_5(\kappa+1)\tau(\tau h^2+\tau^2)\\
& \le\Big(1+\tau+\frac{\tau^2}{2}\Big)\|e_2^n\|_\infty+C\tau(h^2+\tau^2),
\end{align*}
where $C=C_4+\frac{1}{2}C_5(\kappa+1)$.  The condition $\tau\le1$ is used  in the last two steps of the above derivation.
Using the Gronwall's inequality, we obtain
\begin{align}
\|e_2^n\|_\infty
& \le\Big(1+\tau+\frac{\tau^2}{2}\Big)^n\|e_2^0\|_\infty+C\tau(h^2+\tau^2)\sum_{k=0}^{n-1}\Big(1+\tau+\frac{\tau^2}{2}\Big)^k\nonumber\\
& \le\Big(1+\tau+\frac{\tau^2}{2}\Big)^n\|e_2^0\|_\infty+C\Big[\Big(1+\tau+\frac{\tau^2}{2}\Big)^n-1\Big](h^2+\tau^2)\nonumber\\
& \le\e^{n\tau}\|e_2^0\|_\infty+C\e^{n\tau}(h^2+\tau^2),
\end{align}
which then gives us \eqref{erretd2}.
\end{proof}


Now, let us investigate the asymptotic compatibility of both ETD1 and ETDRK2 schemes.
Combining \eqref{continuous_consistency} with
the uniform estimates \eqref{discrete_consistency} of the consistency of $\hL_{\delta,h}$,
we obtain
\begin{equation}
\label{discrete_asympt_consistency}
\max_{\bx_\bi\in\Omega_h}|\hL_{\delta,h}u(\bx_\bi)-\hL_0 u(\bx_\bi)|\le C(\delta^2+h^2)\|u\|_{C^4},
\end{equation}
where $C>0$ is a constant independent of $\delta$ and $h$.
Then, we can further obtain the asymptotic compatibility of the solution to the ETD1 scheme \eqref{nonlocalAC_ETD1}.
Denote by $\varphi(\bx,t)$ the solution of the LAC equation \eqref{localAC}.
For given $\varphi(\bx,t_n)$,
the solution $\varphi(\bx,t_{n+1})$ is determined by $\varphi(\bx,t_{n+1})=v(\bx,\tau)$ with the function $v(\bx,t)$ satisfying
\begin{equation}
\label{localAC_heat}
\begin{dcases}
\piandao{v}{t}=-\kappa v+\eps^2\hL_0v+f(v), & \bx\in\Omega,\ t\in(0,\tau),\\
\text{$v(\cdot,t)$ is $\Omega$-periodic}, & t\in[0,\tau],\\
v(\bx,0)=\varphi(\bx,t_n), & \bx\in\overline{\Omega}.
\end{dcases}
\end{equation}
Let $\widehat{e}(t)=W_1(t)-I^hv(t)$, where $W_1(t)$ is defined by \eqref{nonlocalAC_ETD1_heat}.
Then, the difference between \eqref{nonlocalAC_ETD1_heat} and \eqref{localAC_heat} yields
\begin{equation}
\label{localAC_ETD1_errorheat}
\begin{dcases}
\daoshu{\widehat{e}(t)}{t}=-L_h\widehat{e}(t)+f(U^n)-f(I^h\varphi(t_n))+\widehat{R}^\delta_{h\tau}(t), & t\in(0,\tau),\\
\widehat{e}(0)=U^n-I^h\varphi(t_n)=:\widehat{e}^n,
\end{dcases}
\end{equation}
where the remainder $\widehat{R}^\delta_{h\tau}(t)$ is given by
$$\widehat{R}^\delta_{h\tau}(t)=\eps^2(D_hI^h\varphi(t_n+t)-I^h\hL_0\varphi(t_n+t))+f(I^h\varphi(t_n))-f(I^h\varphi(t_n+t)),$$
and, according to the estimates \eqref{discrete_asympt_consistency}
and the Lipschitz continuity of $f$ under the condition $\kappa\ge 2$,
is bounded by
$$\|\widehat{R}^\delta_{h\tau}(t)\|_\infty\le C(\delta^2+h^2+\tau),\quad\forall\,t\in[0,\tau],$$
where $C>0$ depends on $\eps$, $\kappa$ and the $C^1([0,T];\Cp^4(\overline{\Omega}))$ norm of $\varphi$,
but independent of $\delta$, $h$ and $\tau$.
By conducting similar analysis as done for  error estimates,
we can obtain the asymptotic compatibility of the numerical solution to the ETD1 scheme \eqref{nonlocalAC_ETD1}.
So does the case for the ETDRK2 scheme \eqref{nonlocalAC_ETDRK2}.
Therefore, we have the following results.

\begin{theorem}[Asymptotic compatibility]
Assume that the solution $\varphi$ of the local Allen-Cahn equation \eqref{localAC}
belongs to $C^1([0,T];\Cp^4(\overline{\Omega}))$ {\rm(resp. $C^2([0,T];\Cp^4(\overline{\Omega}))$)}
and $\{U^n\}_{n=0}^{K_t}$ is generated by the ETD1 scheme \eqref{nonlocalAC_ETD1}
{\rm(resp. the ETDRK2 scheme \eqref{nonlocalAC_ETDRK2})}
with $U^0=I^h\varphi_0$.
If $\|\varphi_0\|_{L^\infty}\le1$, then we have
\begin{equation}
\begin{aligned}
\|U^n-I^h\varphi(t_n)\|_\infty & \le C\e^{t_n}(\delta^2+h^2+\tau),\quad t_n\le T\\
\text{\rm(resp. }\|U^n-I^h\varphi(t_n)\|_\infty & \le C\e^{t_n}(\delta^2+h^2+\tau^2),\quad t_n\le T\text{\rm)}
\end{aligned}
\end{equation}
for any $\delta>0$, $h>0$ and $\tau>0$ {\rm(resp. $\tau\in(0,1]$)},
where the constant $C$ depends on the $C^1([0,T];\Cp^4(\overline{\Omega}))$ norm
{\rm(resp. $C^2([0,T];\Cp^4(\overline{\Omega}))$ norm)} of $\varphi$,
but independent of $\delta$, $h$ and $\tau$.
\end{theorem}

\section{Discrete energy stability}

We  first show that the ETD1 scheme \eqref{nonlocalAC_ETD1} inherits  the energy decay law \eqref{energy_decay}
in the discrete sense,
with respect to the discretized energy $E_h$ defined by
\begin{equation}
\label{nonlocal_disenergy}
E_h(U)=\frac{1}{4}\sum_{i=1}^{dN}(U_i^2-1)^2-\frac{\eps^2}{2}U^TD_hU,\quad\forall\,U\in\R^{dN}.
\end{equation}

\begin{theorem}
The approximating solution $\{U^n\}_{n=0}^{K_t}$ generated by the ETD1 scheme \eqref{nonlocalAC_ETD1}
satisfies the energy inequality
$$E_h(U^{n+1})\le E_h(U^n),\quad 0\le n\le K_t-1$$
for any $\tau>0$, i.e., the ETD1 scheme is unconditionally energy stable.
\end{theorem}

\begin{proof}
The difference between the discrete energies at two consecutive time levels yields
\begin{align}
E_h(U^{n+1})-E_h(U^n)
 =&\frac{1}{4}\sum_{i=1}^{dN}[((U^{n+1}_i)^2-1)^2-((U^n_i)^2-1)^2]\nonumber\\
& -\frac{\eps^2}{2}[(U^{n+1})^TD_hU^{n+1}-(U^n)^TD_hU^n].\label{energy stability_pf}
\end{align}
It is easy to verify that
$$\frac{1}{4}[(a^2-1)^2-(b^2-1)^2]\le(b^3-b)(a-b)+2(a-b)^2,\qquad\forall\,a,b\in[-1,1].$$
Since $\kappa\ge 2$,
it follows from Theorem \ref{thm_discrete_max_ETD1} that $\|U^n\|_\infty\le1$ and $\|U^{n+1}\|_\infty\le1$,
then we have
\begin{align*}
&  \frac{1}{4}\sum_{i=1}^{dN}[((U^{n+1}_i)^2-1)^2-((U^n_i)^2-1)^2]\\
&\qquad \le(U^{n+1}-U^n)^T((U^n)^{.3}-U^n)+\kappa(U^{n+1}-U^n)^T(U^{n+1}-U^n)\\
&\qquad =\kappa(U^{n+1}-U^n)^TU^{n+1}-(U^{n+1}-U^n)^Tf(U^n).
\end{align*}
On the other hand, direct calculations lead to
\begin{align*}
&  -\frac{\eps^2}{2}[(U^{n+1})^TD_hU^{n+1}-(U^n)^TD_hU^n]\\
& \qquad =-\eps^2(U^{n+1}-U^n)^TD_hU^{n+1}+\frac{\eps^2}{2}(U^{n+1}-U^n)^TD_h(U^{n+1}-U^n)\\
& \qquad \le -\eps^2(U^{n+1}-U^n)^TD_hU^{n+1}
\end{align*}
due to the negative semi-definiteness of the matrix $D_h$.
Thus, we obtain from \eqref{energy stability_pf} that
\begin{align}
E_h(U^{n+1})-E_h(U^n)
& \le \kappa (U^{n+1}-U^n)^TU^{n+1}-(U^{n+1}-U^n)^Tf(U^n)\nonumber\\
& \qquad -\eps^2(U^{n+1}-U^n)^TD_hU^{n+1}\nonumber\\
& =(U^{n+1}-U^n)^T(L_hU^{n+1}-f(U^n)).
\end{align}
We solve $f(U^n)$ from \eqref{nonlocalAC_ETD1var}, together with Lemma \ref{lem_matfun} (1), to get
\begin{align*}
f(U^n) & =(I-\e^{-L_h\tau})^{-1}L_h(U^{n+1}-\e^{-L_h\tau}U^n)\\
& =(I-\e^{-L_h\tau})^{-1}L_h(U^{n+1}-U^n+(I-\e^{-L_h\tau})U^n)\\
& =(I-\e^{-L_h\tau})^{-1}L_h(U^{n+1}-U^n)+L_hU^n,
\end{align*}
and then
$$L_hU^{n+1}-f(U^n)=L_h(U^{n+1}-U^n)-(I-\e^{-L_h\tau})^{-1}L_h(U^{n+1}-U^n)=B_1(U^{n+1}-U^n),$$
where $B_1:=L_h-(I-\e^{-L_h\tau})^{-1}L_h$.
Define a function
$$g_1(a):=a-\frac{a}{1-\e^{-a}},\quad a\not=0,$$
then $g_1(a)<0$ for any $a>0$ and $B_1=\tau^{-1}g_1(L_h\tau)$.
Since $L_h$ is symmetric and positive definite,
by using Lemma \ref{lem_matfun} (2) and (3),
we know that $B_1$ is symmetric and negative definite.
Therefore, we obtain
$$E_h(U^{n+1})-E_h(U^n)\le(U^{n+1}-U^n)^TB_1(U^{n+1}-U^n)\le0,$$
which completes the proof.
\end{proof}

For the ETDRK2 scheme \eqref{nonlocalAC_ETDRK2},
we can prove the uniform boundedness of the discretized energy $E_h$.

\begin{theorem}
Under the assumptions of Theorem \ref{thm_error_ETDRK2},
the approximate solution $\{U^n\}_{n=0}^{K_t}$ generated by the ETDRK2 scheme \eqref{nonlocalAC_ETDRK2} satisfies
$$E_h(U^{n+1})\le E_h(U^n)+\widetilde{C}h^{-\frac{1}{2}}(h^2+\tau)^2,\quad 0\le n\le K_t-1$$
for any $h>0$ and $1\ge \tau>0$, where the constant $\widetilde{C}$ is independent of $h$ and $\tau$.
Furthermore, if $h\le1$ and $\tau=\lambda\sqrt{h}$ for some constant $\lambda>0$, we have
$$E_h(U^n)\le E_h(U^0)+\widehat{C},\quad 0\le n\le K_t,$$
where the constant $\widehat{C}$ is independent of $h$ and $\tau$,
i.e., the discrete energy is uniformly bounded.
\end{theorem}

\begin{proof}
The first step is to calculate the increment $E_h(U^{n+1})-E_h(U^n)$ directly,
which is completely identical to the proof for the ETD1 scheme,
and we obtain
$$E_h(U^{n+1})-E_h(U^n)\le(U^{n+1}-U^n)^T(L_hU^{n+1}-f(U^n)).$$
Using \eqref{nonlocalAC_ETDRK2var1} and \eqref{nonlocalAC_ETDRK2var2}, we then get
\begin{equation}
\label{ttt}
U^{n+1}=\phi_0(L_h\tau)U^n+\tau\phi_1(L_h\tau)f(U^n)+\tau\phi_2(L_h\tau)(f(\widetilde{U}^{n+1})-f(U^n)).
\end{equation}
Acting $(\tau\phi_1(L_h\tau))^{-1}$ on both sides of \eqref{ttt} gives us
\begin{align*}
f(U^n) & =(I-\e^{-L_h\tau})^{-1}L_h(U^{n+1}-U^n)+L_hU^n\\
& \qquad -(L_h\tau)^{-1}(I-\e^{-L_h\tau})^{-1}(\e^{-L_h\tau}-I+L_h\tau)(f(\widetilde{U}^{n+1})-f(U^n)),
\end{align*}
and then, using the notation $B_1$ defined in the proof for the ETD1 scheme,
we obtain
$$L_hU^{n+1}-f(U^n)=B_1(U^{n+1}-U^n)+B_2(f(\widetilde{U}^{n+1})-f(U^n)),$$
where $B_2=g_2(L_h\tau)$ with
$$g_2(a):=\frac{e^{-a}-1+a}{a(1-\e^{-a})},\quad a\not=0.$$
Since $0<g_2(a)<1$ for any $a>0$, we know that $B_2$ is symmetric and positive definite and $\|B_2\|_2<1$.
Using the mean-value theorem, we have
$$f(\widetilde{U}^{n+1})-f(U^n)=G^n(\widetilde{U}^{n+1}-U^n),$$
where $G^n$ is a diagonal matrix with, according to \eqref{lem_rhsnorm_f0p}, the diagonal entries between $0$ and $\kappa+1$,
which implies that $\|G^n\|_\infty\le \kappa+1$.
Then, by the negative definiteness of $B_1$, we obtain
$$E_h(U^{n+1})-E_h(U^n)\le(U^{n+1}-U^n)^TB_2G^n(\widetilde{U}^{n+1}-U^n).$$
According to Theorem \ref{thm_error_ETDRK2}, we can derive
\begin{align*}
\|U^{n+1}-U^n\|_\infty
& \le\|U^{n+1}-I^hu(t_{n+1})\|_\infty+\|I^h(u(t_{n+1})-u(t_n))\|_\infty+\|I^hu(t_n)-U^n\|_\infty\\
& \le C_6\e^{t_{n+1}}(h^2+\tau^2)+C_7\tau+C_6\e^{t_n}(h^2+\tau^2)\\
& \le C_8(h^2+\tau),
\end{align*}
for any $h>0$ and $\tau\in(0,1]$,
where the constants $C_6$ and $C_7$ depend on the $C^2([0,T];\Cp^4(\overline{\Omega}))$ norm of $u$
and $C_8=2C_6\e^T+C_7$.
Similarly, using Theorems \ref{thm_error_ETD1}, we can obtain
$$\|\widetilde{U}^{n+1}-U^n\|_\infty\le C_9(h^2+\tau),$$
with the constant $C_9>0$ depending on $T$ and the $C^1([0,T];\Cp^4(\overline{\Omega}))$ norm of $u$.
In addition, we have
$$\|B_2G^n\|_\infty\le\sqrt{dN}\|B_2\|_2\|G^n\|_\infty\le\sqrt{dN}(\kappa+1)=h^{-\frac{1}{2}}\sqrt{dX}(\kappa+1).$$
Therefore, we obtain
$$E_h(U^{n+1})-E_h(U^n)\le\widetilde{C}h^{-\frac{1}{2}}(h^2+\tau)^2,$$
where the constant $\widetilde{C}=C_8C_9\sqrt{dX}(\kappa+1)$ is independent of $h$ and $\tau$.
By induction, we further obtain
$$E_h(U^n)\le E_h(U^0)+\widetilde{C}Th^{-\frac{1}{2}}\tau^{-1}(h^2+\tau)^2,\quad 0\le n\le K_t.$$
When $h\le1$ and $\tau=\lambda\sqrt{h}$, it holds $h^{-\frac{1}{2}}\tau^{-1}(h^2+\tau)^2\le(\lambda+1)^2/\lambda$.
The proof is then completed by setting $\widehat{C}=\widetilde{C}T(\lambda+1)^2/\lambda$.
\end{proof}


\section{Numerical experiments}

In this section, we will carry out some numerical experiments in the 2D space
to demonstrate the effectiveness and efficiency of the ETD schemes \eqref{nonlocalAC_ETD1} and \eqref{nonlocalAC_ETDRK2}
for solving the NAC equation \eqref{nonlocalAC}.
The fractional power kernels
\begin{equation}
\label{frac_kernel}
\rho_\delta(r)=\frac{2(4-\alpha)}{\pi\delta^{4-\alpha}r^\alpha}\chi_{(0,\delta]}(r),\quad\alpha\in[0,4),
\end{equation}
is chosen and satisfies the finite second order moment condition \eqref{kernel_2ndmoment}.
When $\alpha\in[0,2)$, the kernel satisfies $\rho_\delta(|\bs|)\in L^1(\R^2)$,
which means that the nonlocal diffusion operator $\hL_\delta$ is a bounded linear operator in this case.
 The kernel is non-integrable when $\alpha\in[2,4)$.
We first verify the temporal and spatial convergence rates of the fully discrete schemes with a smooth initial data,
and then check the discrete maximum principle and energy stability of the evolutions beginning with a random initial state.
Next, we present a further numerical investigation on the steady state solutions to the model with integrable kernels.
The ETDRK2 scheme \eqref{nonlocalAC_ETDRK2} is adopted in all the simulations
while the ETD1 scheme \eqref{nonlocalAC_ETD1} is only considered for the temporal convergence tests due to the lack of high accuracy.
The domain $\Omega=(0,2\pi)\times(0,2\pi)$ is used in all examples.
We also take the stabilizing parameter  $\kappa=2$ for the numerical schemes in all experiments.

\subsection{Convergence tests}
~

\begin{example}
\label{eg_convergence}
We consider the NAC equation \eqref{nonlocalAC} with a smooth initial data $u_0(x,y)=0.5\sin x\sin y$.
We set the interfacial parameter $\eps=0.1$ and the terminal time $T=0.5$.
For the kernel \eqref{frac_kernel}, $\alpha=1$ (integrable) and $\alpha=3$ (non-integrable) are adopted, respectively.
\end{example}

First, by setting $N=256$, we tested the convergence in time for the cases $\delta=0.2$ and $\delta=2$.
We calculated the numerical solutions of the NAC equation
using the ETD1 scheme \eqref{nonlocalAC_ETD1} and the ETDRK2 scheme \eqref{nonlocalAC_ETDRK2}
with various time step sizes $\tau=0.05\times2^{-k}$ with $k=0,1,\dots,7$.
To compute the errors, we treated the solution obtained by the
ETDRK2 scheme with $\tau=10^{-6}$ as the benchmark.
The maximum-norms of the numerical errors and corresponding convergence rates are given in \tablename~\ref{table_conv_time},
where the expected temporal convergence rates ($1$ for ETD1 and $2$ for ETDRK2)
are obviously observed in both cases of integrable and non-integrable kernels.
It is also easy to see that the numerical errors are almost independent of the choices of $\delta$ and $\alpha$.

\begin{table}[!ht]
\centering
\caption{ Temporal convergence rates in the maximum-norm sense in Example \ref{eg_convergence}.}
\label{table_conv_time}
\small
\begin{tabular}{|*{10}{c|}}
\hline
\multicolumn{2}{|c|}{\multirow{3}{*}{$\tau=0.05$}}
& \multicolumn{4}{c|}{$\alpha=1$ (integrable kernel)} & \multicolumn{4}{c|}{$\alpha=3$ (non-integrable kernel)} \\
\cline{3-10}
\multicolumn{2}{|c|}{} & \multicolumn{2}{c|}{$\delta=0.2$} & \multicolumn{2}{c|}{$\delta=2$} & \multicolumn{2}{c|}{$\delta=0.2$} & \multicolumn{2}{c|}{$\delta=2$} \\
\cline{3-10}
\multicolumn{2}{|c|}{} & Error & Rate & Error & Rate & Error & Rate & Error & Rate \\
\hline\hline
\multirow{8}{*}{ETD1} & $\tau$ & 1.082e-2 & $-$ & 1.090e-2 & $-$ & 1.084e-2 & $-$ & 1.087e-2 & $-$ \\
~ & $\tau/2$ & 5.535e-3 & 0.9670 & 5.580e-3 & 0.9666 & 5.545e-3 & 0.9669 & 5.561e-3 & 0.9668 \\
~ & $\tau/4$ & 2.800e-3 & 0.9833 & 2.823e-3 & 0.9831 & 2.805e-3 & 0.9833 & 2.813e-3 & 0.9832 \\
~ & $\tau/8$ & 1.408e-3 & 0.9916 & 1.420e-3 & 0.9915 & 1.410e-3 & 0.9916 & 1.415e-3 & 0.9915 \\
~ & $\tau/16$ & 7.060e-4 & 0.9957 & 7.121e-4 & 0.9957 & 7.074e-4 & 0.9957 & 7.095e-4 & 0.9957 \\
~ & $\tau/32$ & 3.536e-4 & 0.9979 & 3.566e-4 & 0.9978 & 3.542e-4 & 0.9979 & 3.553e-4 & 0.9979 \\
~ & $\tau/64$ & 1.769e-4 & 0.9989 & 1.784e-4 & 0.9989 & 1.772e-4 & 0.9989 & 1.778e-4 & 0.9989 \\
~ & $\tau/128$ & 8.849e-5 & 0.9995 & 8.924e-5 & 0.9995 & 8.865e-5 & 0.9995 & 8.892e-5 & 0.9995 \\
\hline
\multirow{8}{*}{ETDRK2} & $\tau$ & 6.410e-4 & $-$ & 6.464e-4 & $-$ & 6.422e-4 & $-$ & 6.441e-4 & $-$ \\
~ & $\tau/2$ & 1.676e-4 & 1.9352 & 1.690e-4 & 1.9350 & 1.679e-4 & 1.9352 & 1.684e-4 & 1.9351 \\
~ & $\tau/4$ & 4.287e-5 & 1.9672 & 4.323e-5 & 1.9671 & 4.294e-5 & 1.9672 & 4.308e-5 & 1.9671 \\
~ & $\tau/8$ & 1.084e-5 & 1.9835 & 1.093e-5 & 1.9835 & 1.086e-5 & 1.9835 & 1.089e-5 & 1.9835 \\
~ & $\tau/16$ & 2.726e-6 & 1.9917 & 2.749e-6 & 1.9917 & 2.730e-6 & 1.9917 & 2.739e-6 & 1.9917 \\
~ & $\tau/32$ & 6.834e-7 & 1.9959 & 6.892e-7 & 1.9959 & 6.846e-7 & 1.9959 & 6.867e-7 & 1.9959 \\
~ & $\tau/64$ & 1.711e-7 & 1.9981 & 1.725e-7 & 1.9981 & 1.714e-7 & 1.9981 & 1.719e-7 & 1.9981 \\
~ & $\tau/128$ & 4.278e-8 & 1.9997 & 4.314e-8 & 1.9997 & 4.285e-8 & 1.9997 & 4.299e-8 & 1.9997 \\
\hline
\end{tabular}
\end{table}

Next, we tested the convergence with respect to the spatial size $h$ by fixing $\delta=2$ and $\tau=T$.
The numerical solution of the NAC equation obtained by the ETDRK2 scheme with $N=4096$ is treated as the benchmark
for computing the errors of the numerical solutions obtained with $N=2^k$ with $k=4,5,\dots,10$.
The numerical errors in the maximum-norm sense are presented in \tablename~\ref{table_conv_space},
and it is observed that the convergence rates with respect to $h$
are almost of second order in both cases of integrable and non-integrable kernels,
which is again consistent with the theoretical results.

\begin{table}[!ht]
\centering
\caption{Spatial convergence rates in the maximum-norm sense in Example \ref{eg_convergence}.}
\label{table_conv_space}
\small
\begin{tabular}{|*{5}{c|}}
\hline
\multirow{2}{*}{$\ds h=\frac{\pi}{8}$} & \multicolumn{2}{c|}{$\alpha=1$} & \multicolumn{2}{c|}{$\alpha=3$} \\
\cline{2-5}
~ & Error & Rate & Error & Rate \\
\hline\hline
$h$ & 1.554e-4 & $-$ & 1.258e-4 & $-$ \\
$h/2$ & 2.430e-5 & 2.6774 & 3.328e-5 & 1.9182 \\
$h/4$ & 4.491e-6 & 2.4358 & 7.441e-6 & 2.1608 \\
$h/8$ & 8.679e-7 & 2.3714 & 2.017e-6 & 1.8833 \\
$h/16$ & 2.068e-7 & 2.0694 & 4.701e-7 & 2.1011 \\
$h/32$ & 4.944e-8 & 2.0643 & 1.344e-7 & 1.8066 \\
$h/64$ & 6.590e-9 & 2.9075 & 3.483e-8 & 1.9479 \\
\hline
\end{tabular}
\end{table}

We also investigated the limit behaviors of the numerical solutions of \eqref{nonlocalAC} as $\delta\to0$.
By fixing $N=4096$ and $\tau=T$,
we calculated the numerical solutions of the NAC equation
obtained by the ETDRK2 scheme \eqref{nonlocalAC_ETDRK2} with various $\delta$'s
and compared them with the numerical solution of the LAC equation.
\tablename~\ref{table_conv_horizon} collects the errors between the nonlocal and local numerical solutions in the maximum-norm sense
and the second order convergence with respect to $\delta$ is obviously observed.

\begin{table}[!ht]
\centering
\caption{Rates of convergence to the local limits in the maximum-norm sense in Example \ref{eg_convergence}.}
\label{table_conv_horizon}
\small
\begin{tabular}{|*{5}{c|}}
\hline
\multirow{2}{*}{$\delta=0.2$} & \multicolumn{2}{c|}{$\alpha=1$} & \multicolumn{2}{c|}{$\alpha=3$} \\
\cline{2-5}
~ & Error & Rate & Error & Rate \\
\hline\hline
$\delta$ & 1.076e-5 & $-$ & 5.371e-6 & $-$ \\
$\delta/2$ & 2.703e-6 & 1.9927 & 1.344e-6 & 1.9991 \\
$\delta/4$ & 6.250e-7 & 2.1124 & 3.153e-7 & 2.0912 \\
$\delta/8$ & 1.580e-7 & 1.9835 & 6.373e-8 & 2.3068 \\
\hline
\end{tabular}
\end{table}

\subsection{Stability tests}

For the case $\rho_\delta(|\bs|)\in L^1(\R^2)$, i.e., $\alpha\in[0,2)$,
it has been proved in \cite{DuYa16} that
the steady state solution $u^*$ to the NAC equation \eqref{nonlocalAC} is continuous if $\eps^2C_\delta\ge1$, where
$$C_\delta=\int_{B_\delta(\bo)}\rho_\delta(|\bs|)\,\d\bs=\frac{4(4-\alpha)}{(2-\alpha)\delta^2}.$$
Under certain assumptions, if $\eps^2C_\delta<1$,
the locally increasing $u^*$ has a discontinuity at $x_*$ with the jump
\begin{equation}
\label{steady_state_jump}
\llbracket u^*\rrbracket(x_*)=2\sqrt{1-\eps^2C_\delta}.
\end{equation}

\begin{example}
\label{eg_stability}
We simulate the NAC equation \eqref{nonlocalAC}
with a random initial data  ranging from $-0.9$ to $0.9$ uniformly generated on the $512\times512$ mesh.
We set the interfacial parameter $\eps=0.1$ and adopt the kernel \eqref{frac_kernel} with $\alpha=1$ and various $\delta$'s.
For the comparison, we also simulate the LAC equation \eqref{localAC} with the same settings.
The time step is set to be $\tau=0.01$ for all cases.
\end{example}

Under these settings, the critical value of $\delta$ to satisfy $\eps^2C_\delta=1$ is $\delta_0=2\sqrt{3}\eps$.
The three rows in \figurename~\ref{fig_stability_1} correspond to
the evolutions of phase structures governed by the LAC equation and the NAC equation
with $\delta=3\eps$ and $\delta=4\eps$ at times $t=6$, $14$, $50$, and $180$, respectively.
\figurename~\ref{fig_stability_2} presents the evolutions of the corresponding maximum-norms
and the energies of the numerical solutions, respectively.
It is observed in all cases that the discrete maximum principle is preserved perfectly and the discrete energy decays monotonically.
It is easy to see that the dynamics of the NAC equation with $\delta=3\eps$ is quite similar to that of the LAC equation.
The evolution processes of these two cases reach the steady states at about $t=190$ and $t=370$, respectively,
while the evolution of the NAC equation with $\delta=4\eps$ lasts much longer time.
In addition, the NAC equation with $\delta=3\eps$ has thinner and sharper interface than the LAC equation
but wider interface than the NAC equation equation with $\delta=4\eps$.
Actually, the interface in the case $\delta=4\eps$ is discontinuous since the condition $\eps^2C_\delta<1$ holds.
The discontinuities in the solutions will be investigated further in the next example.

\begin{figure}[!ht]
\centerline{
\includegraphics[width=0.22\textwidth]{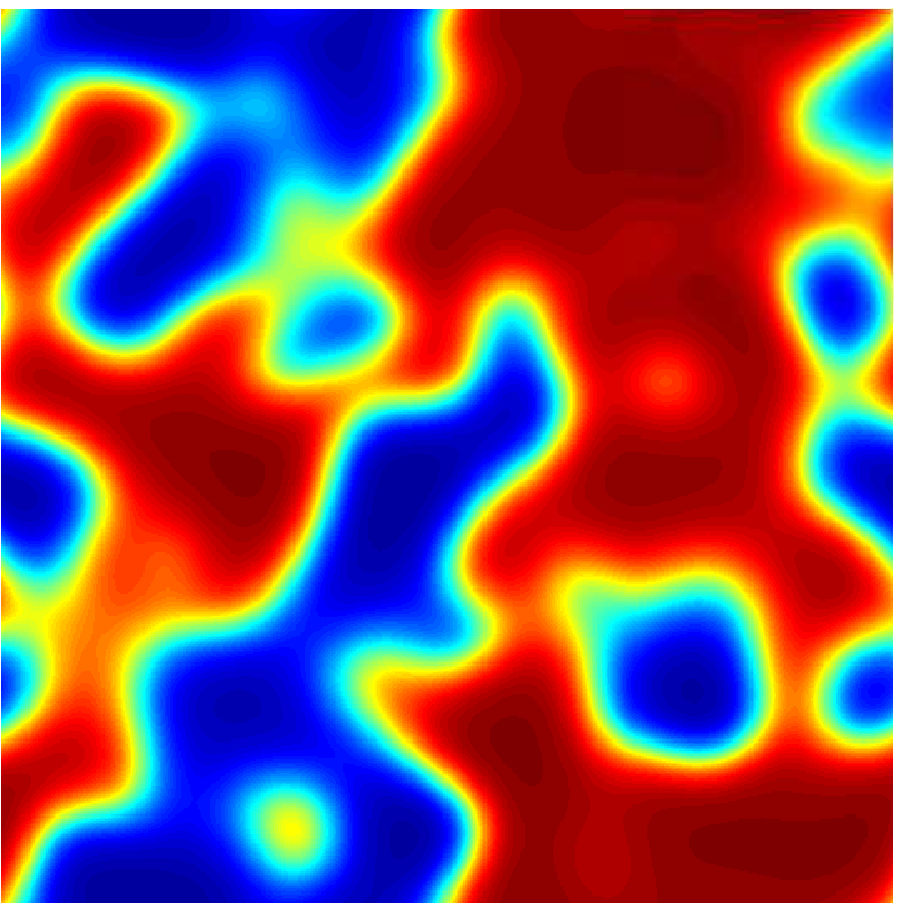}\hspace{0.05cm}
\includegraphics[width=0.22\textwidth]{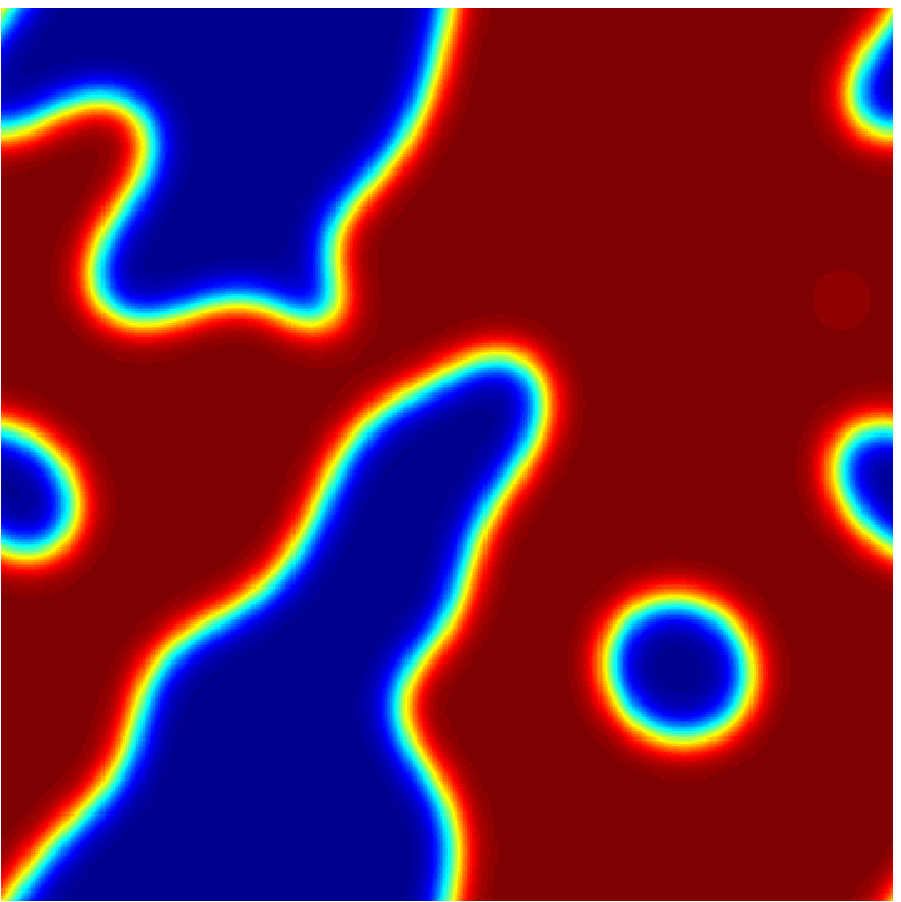}\hspace{0.05cm}
\includegraphics[width=0.22\textwidth]{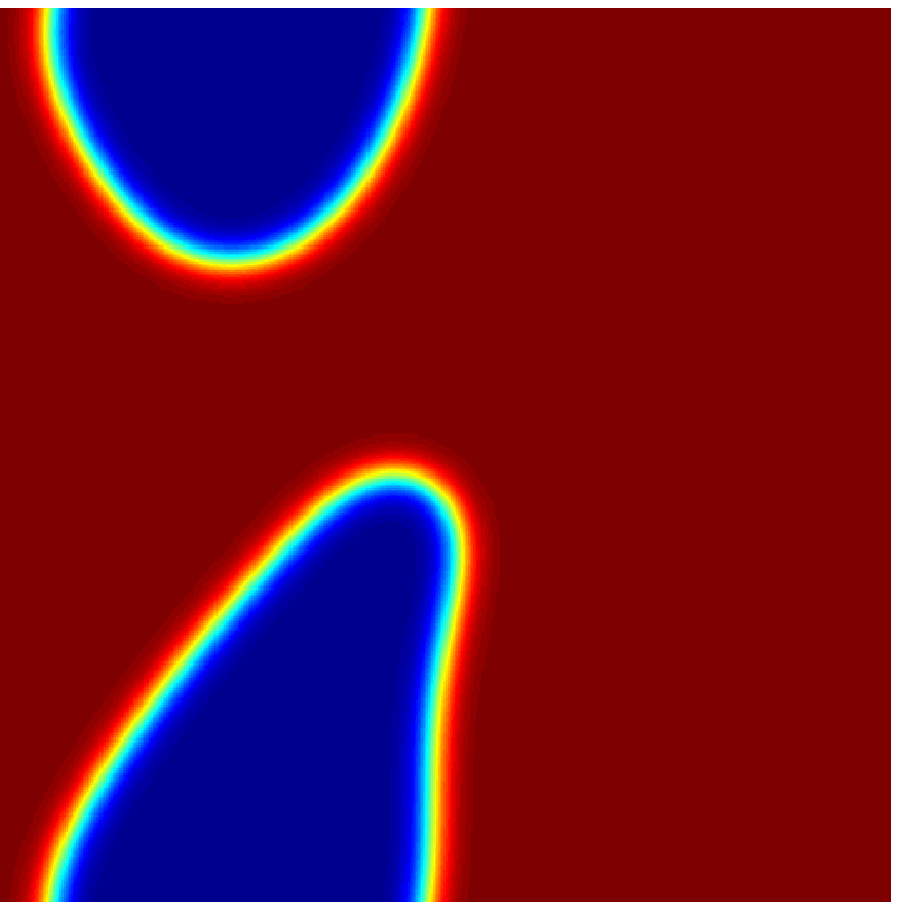}\hspace{0.05cm}
\includegraphics[width=0.22\textwidth]{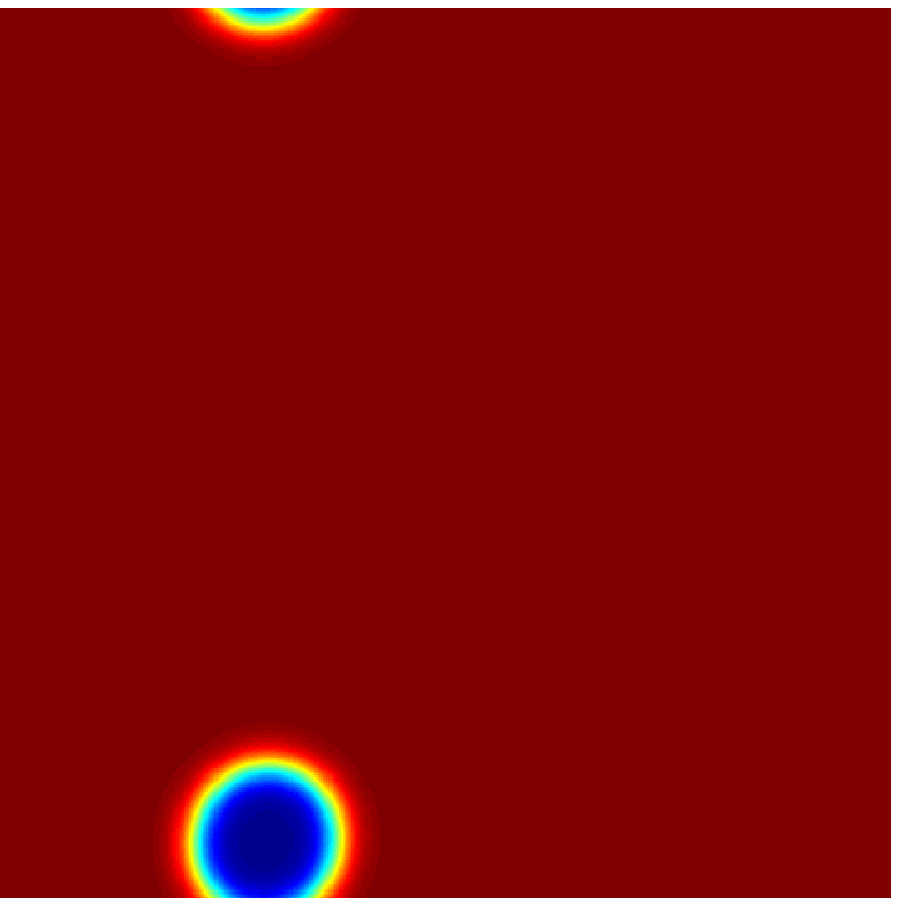}}\vspace{0.2cm}
\centerline{
\includegraphics[width=0.22\textwidth]{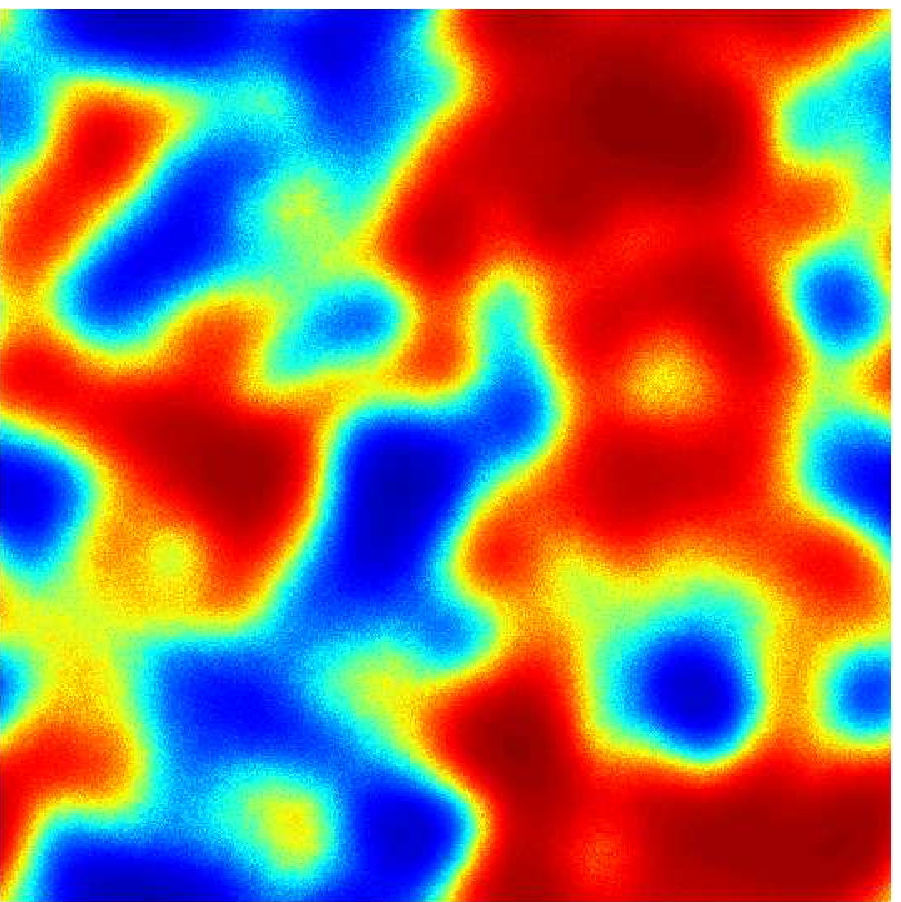}\hspace{0.05cm}
\includegraphics[width=0.22\textwidth]{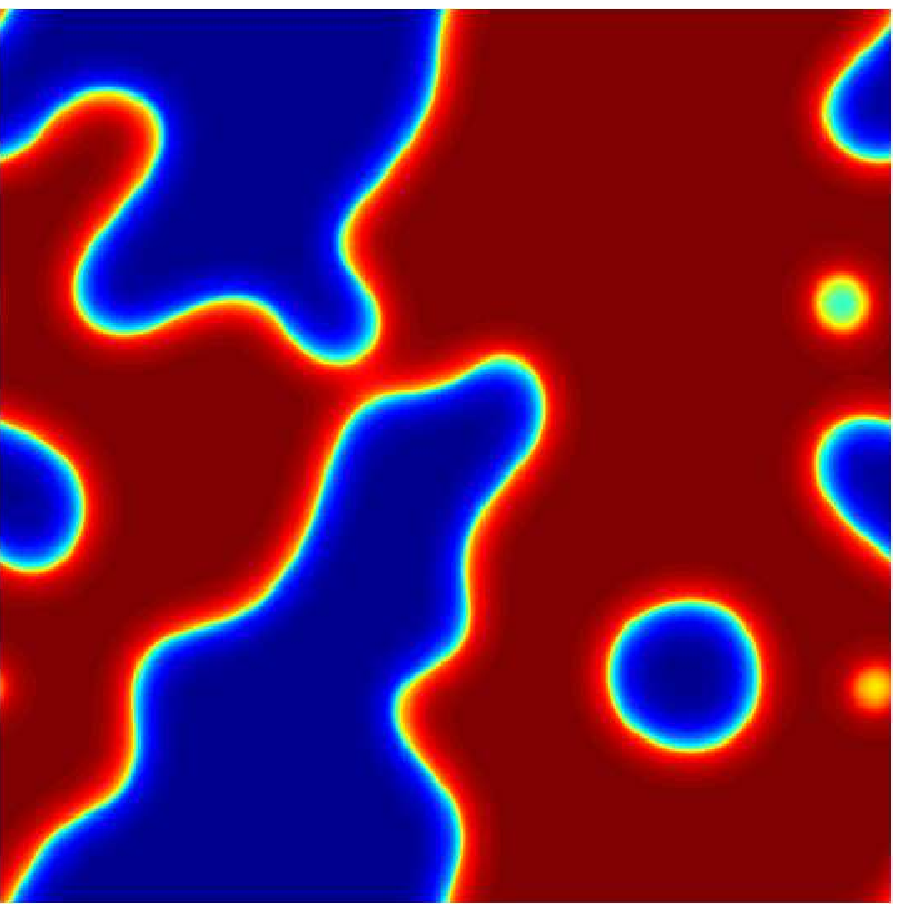}\hspace{0.05cm}
\includegraphics[width=0.22\textwidth]{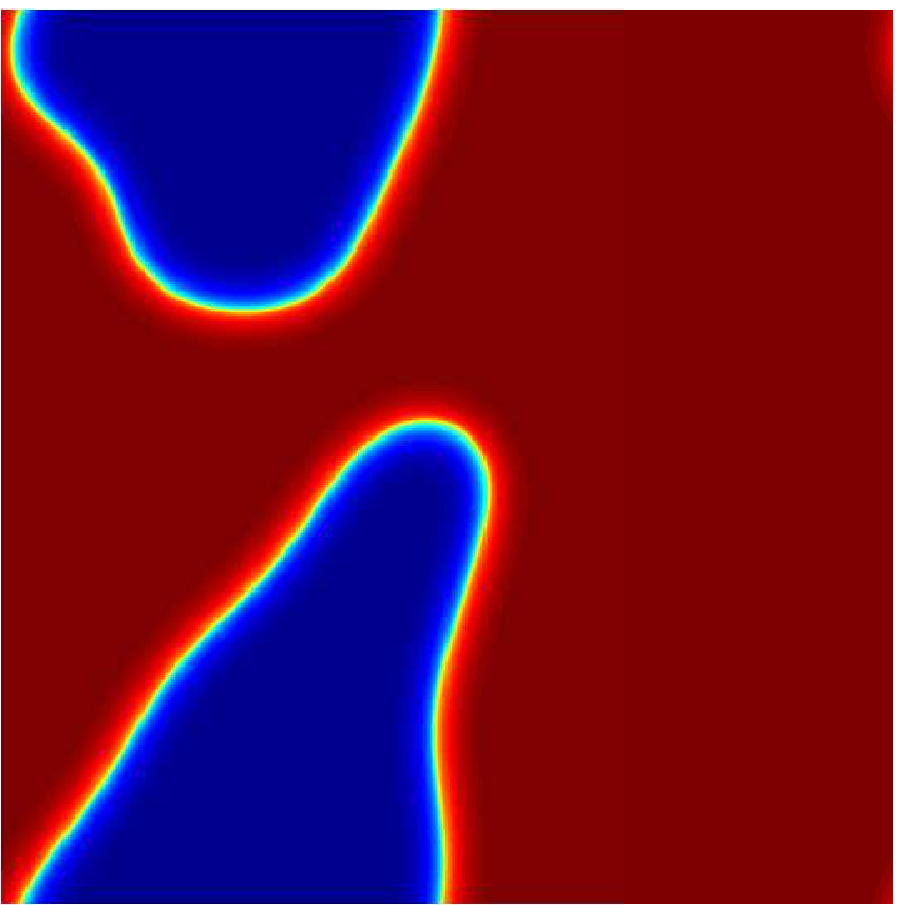}\hspace{0.05cm}
\includegraphics[width=0.22\textwidth]{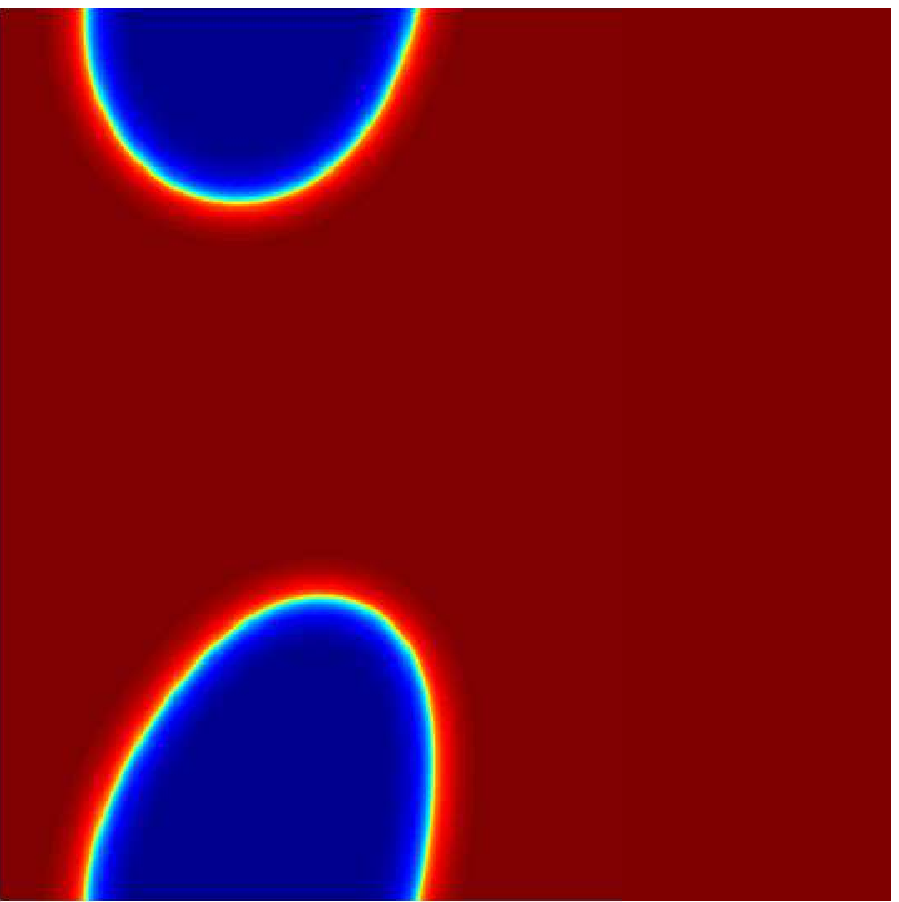}}\vspace{0.2cm}
\centerline{
\includegraphics[width=0.22\textwidth]{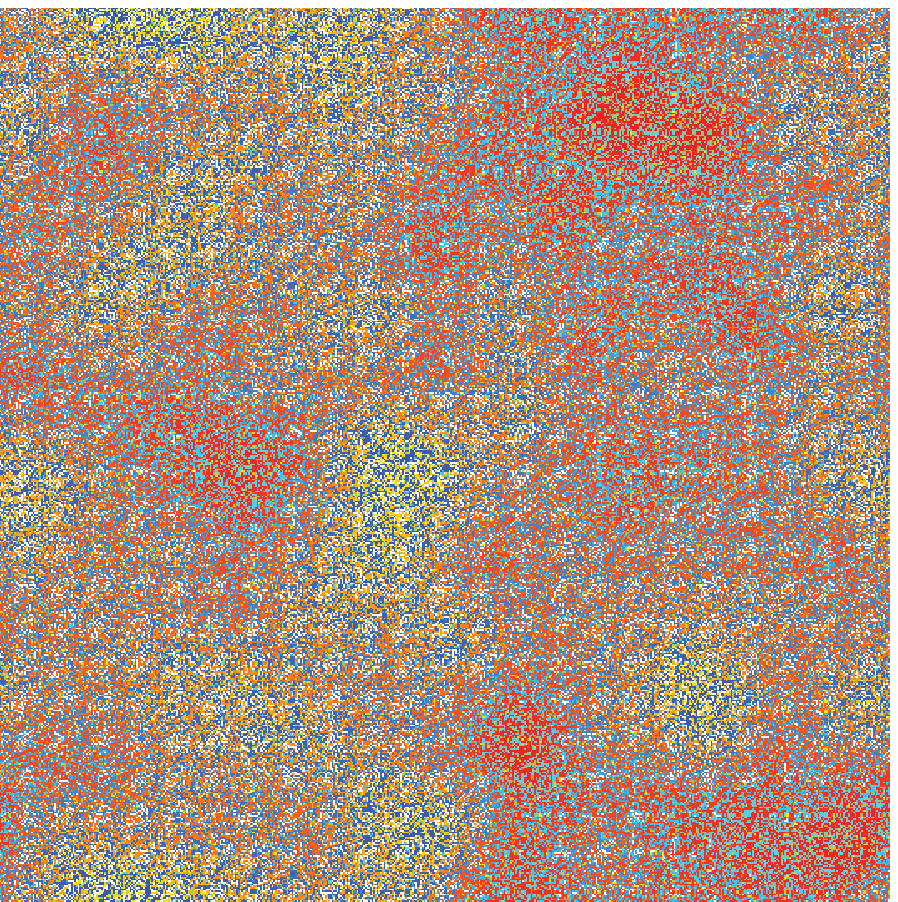}\hspace{0.05cm}
\includegraphics[width=0.22\textwidth]{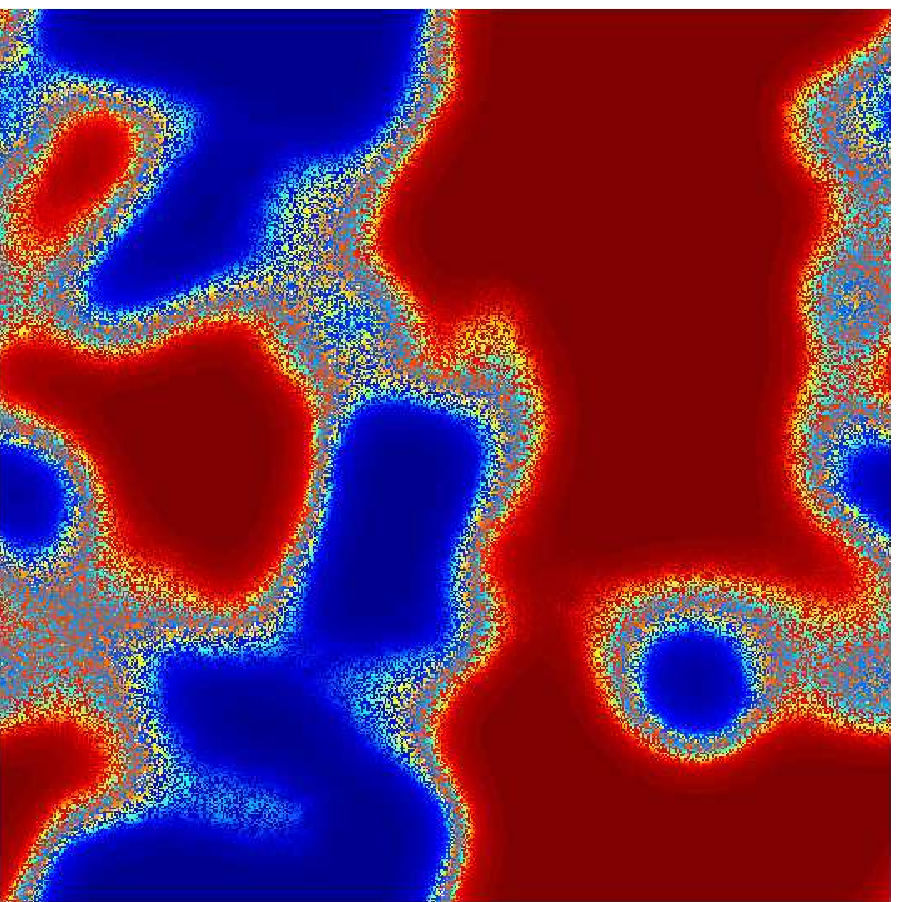}\hspace{0.05cm}
\includegraphics[width=0.22\textwidth]{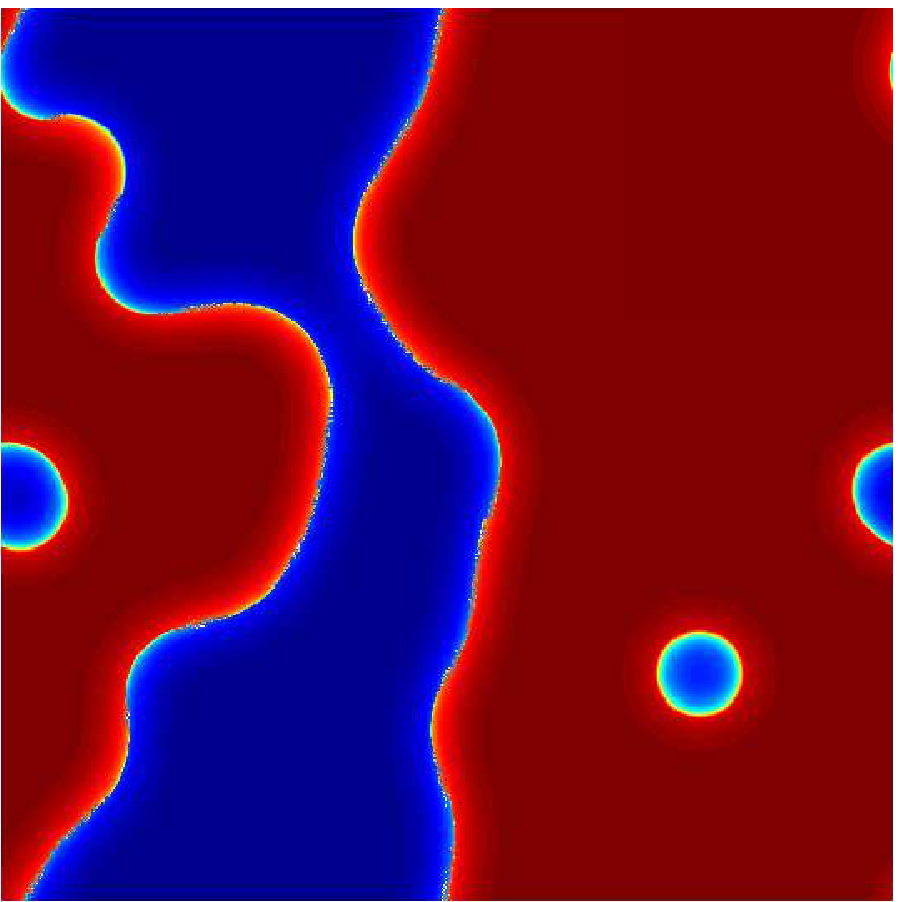}\hspace{0.05cm}
\includegraphics[width=0.22\textwidth]{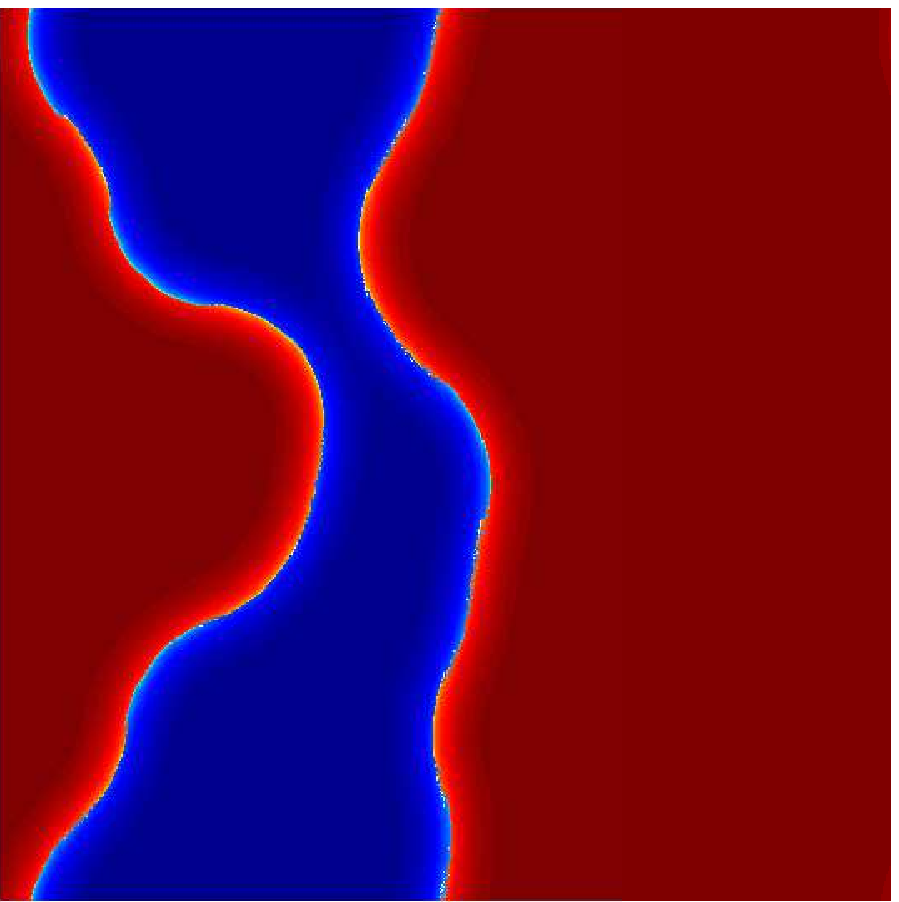}}
\caption{Evolutions of phase structures governed by the LAC equation (top row)
and the NAC equation with $\delta=3\eps$ (middle row) and $\delta=4\eps$ (bottom row) in Example \ref{eg_stability}.
From left to right: $t=6,14,50,180$.}
\label{fig_stability_1}
\end{figure}

\begin{figure}[!ht]
\centerline{
{\includegraphics[width=0.33\textwidth]{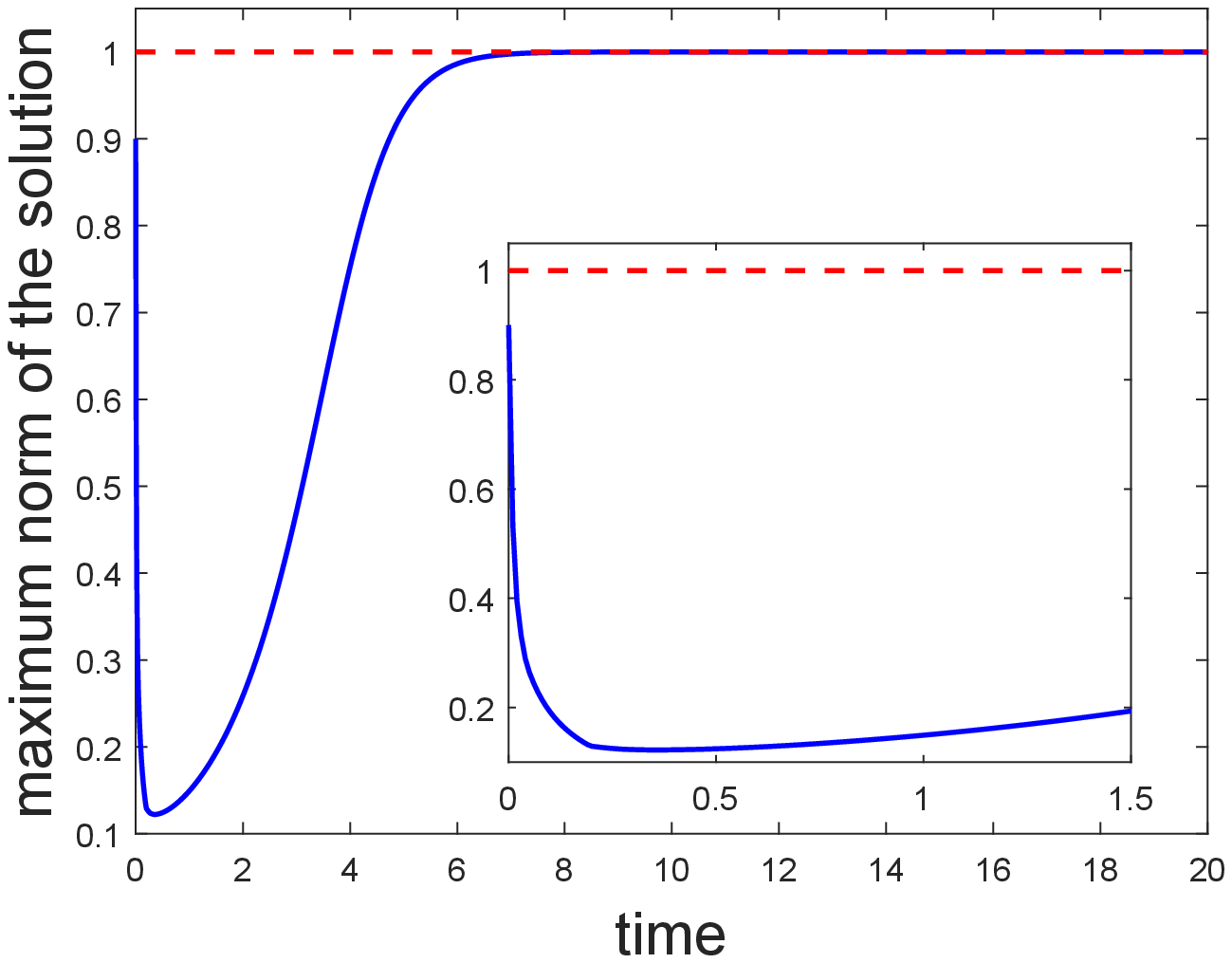}}
{\includegraphics[width=0.33\textwidth]{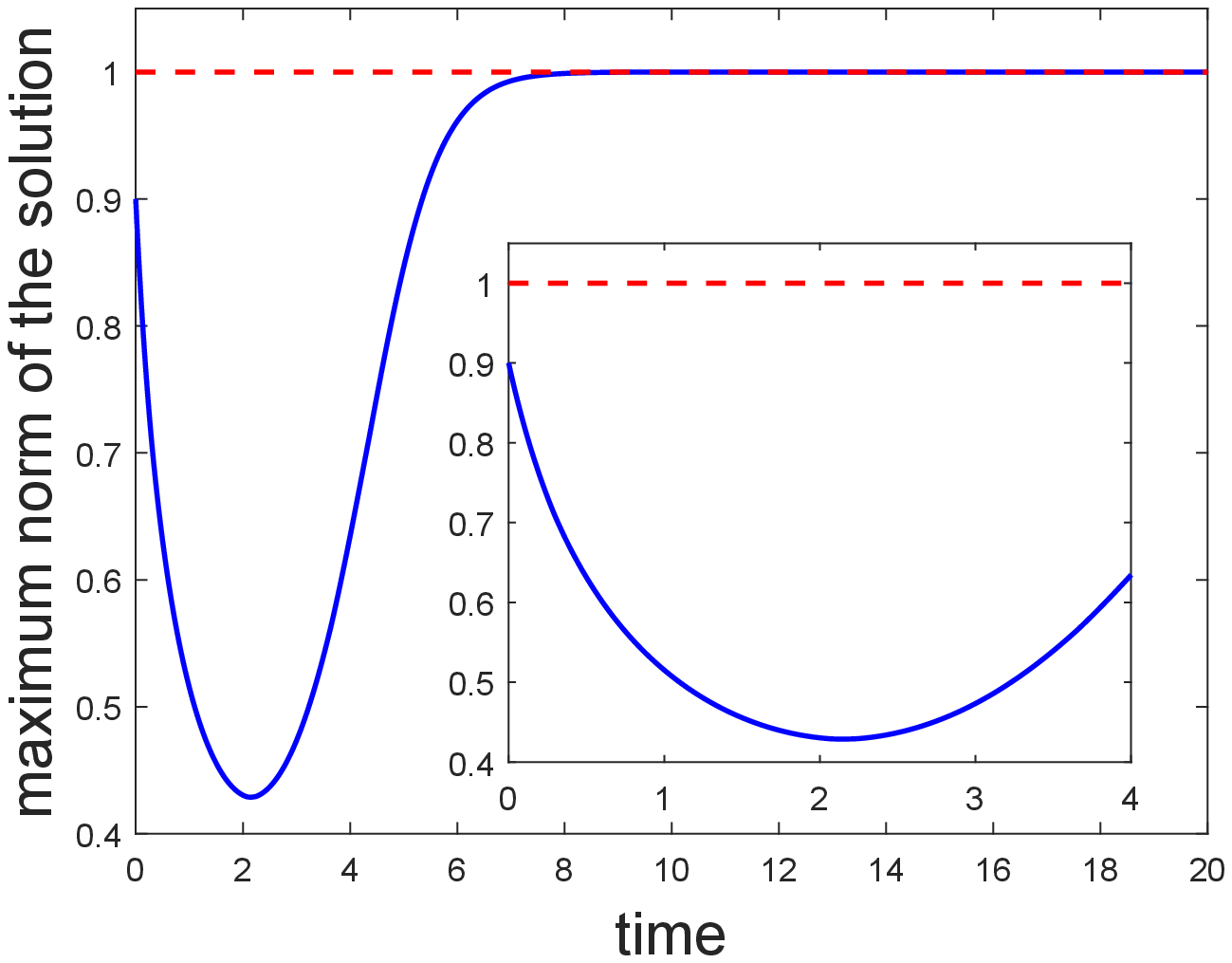}}
{\includegraphics[width=0.33\textwidth]{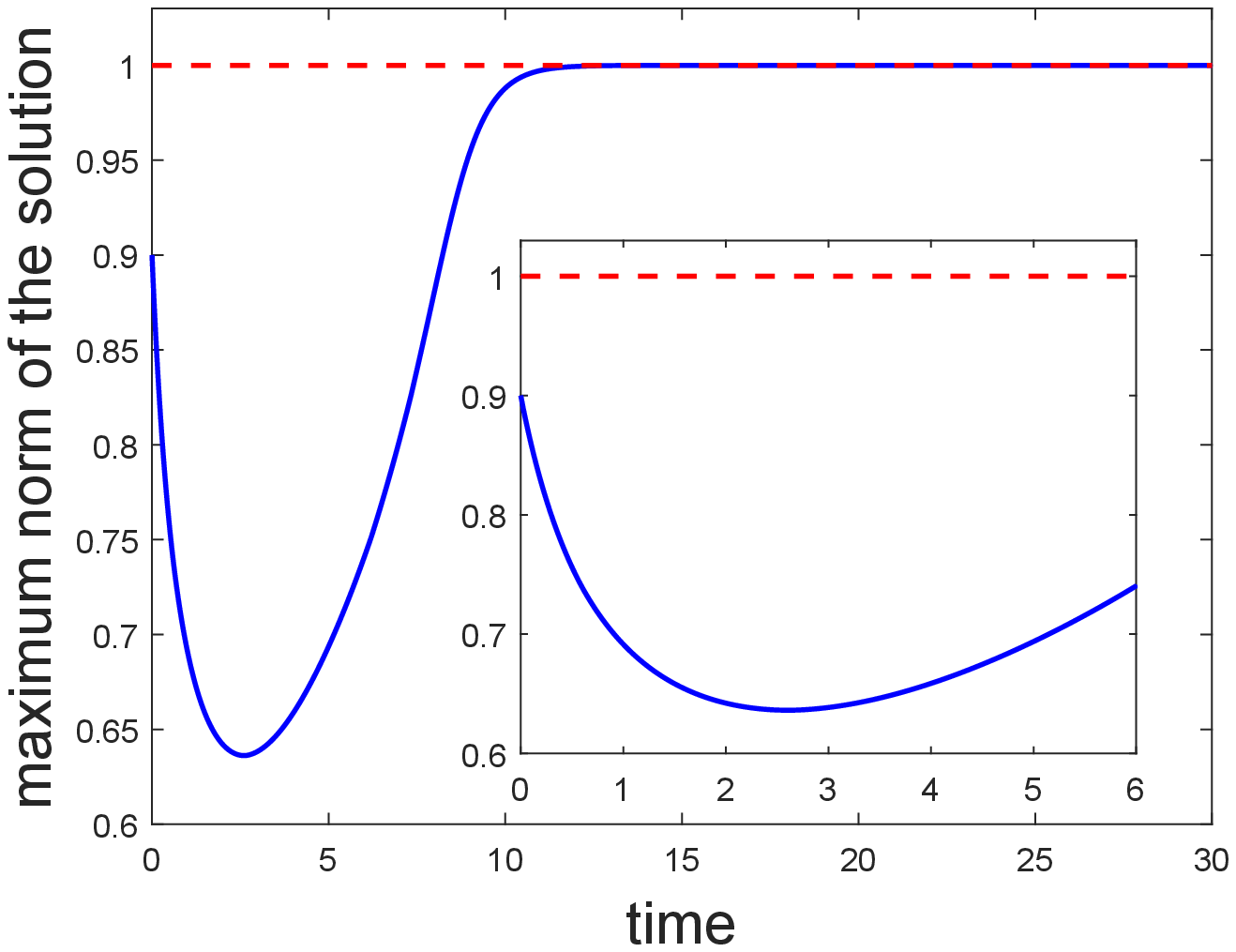}}}
\centerline{
{\includegraphics[width=0.33\textwidth]{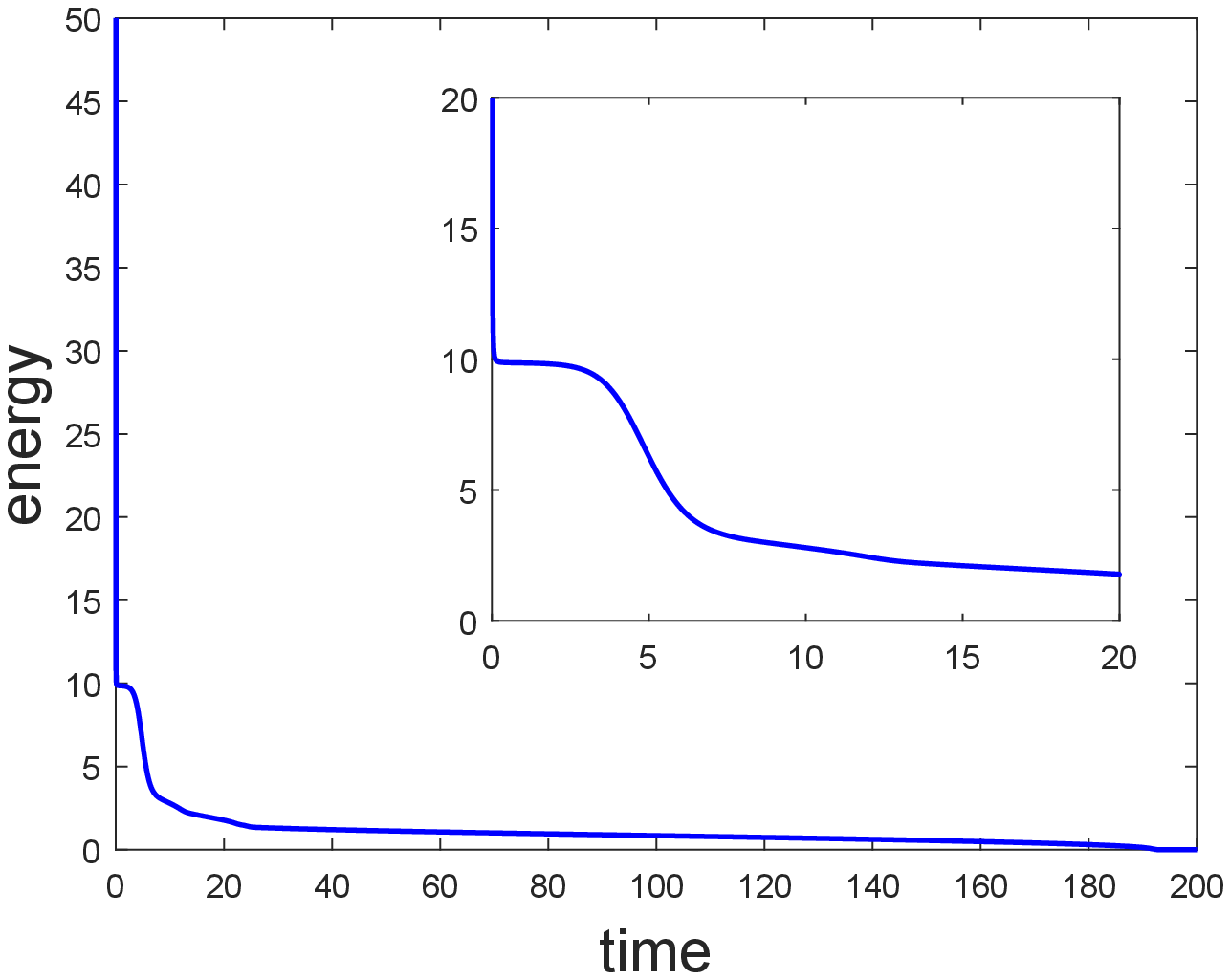}}
{\includegraphics[width=0.33\textwidth]{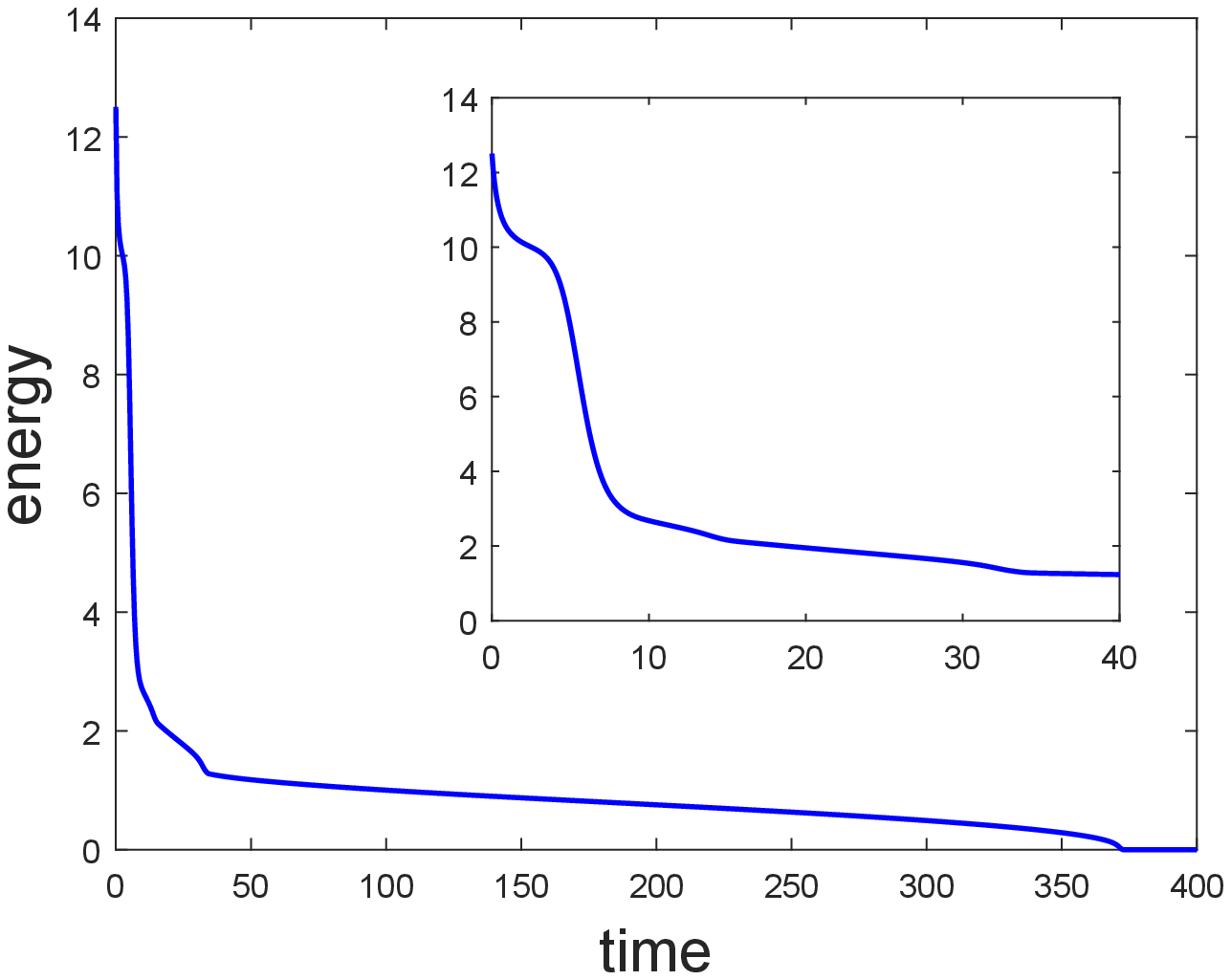}}
{\includegraphics[width=0.33\textwidth]{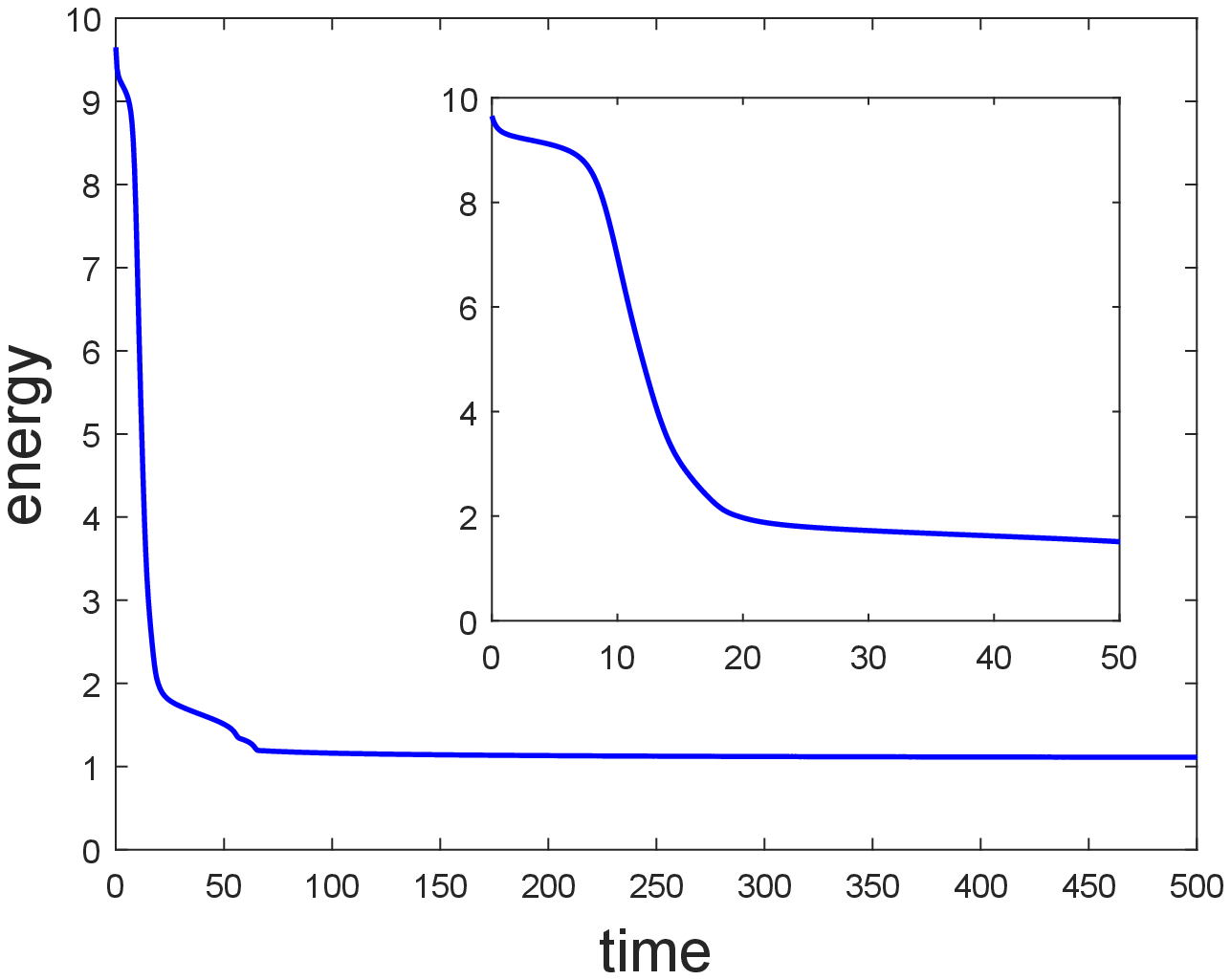}}}
\caption{Evolutions of the maximum norms (top row) and the energies of the numerical solutions in Example \ref{eg_stability}.
From left to right: governed by the LAC equation and the NAC equation with $\delta=3\eps$ and $\delta=4\eps$.}
\label{fig_stability_2}
\end{figure}

\subsection{Discontinuity in the steady state solution}
~

\begin{example}
\label{eg_bubble}
We simulate the evolution of a bubble governed by the NAC equation \eqref{nonlocalAC}
starting with a smooth initial configuration (See \figurename~\ref{fig_bubble_initial}).
Again, we set the interfacial parameter $\eps=0.1$ and adopt the kernel \eqref{frac_kernel} with $\alpha=1$ and various $\delta$'s.
The parameters of space-time mesh are set to be $\tau=0.01$ and $N=2048$ for all cases.
\end{example}

\begin{figure}[!htp]
\centerline{
{\includegraphics[width=0.45\textwidth]{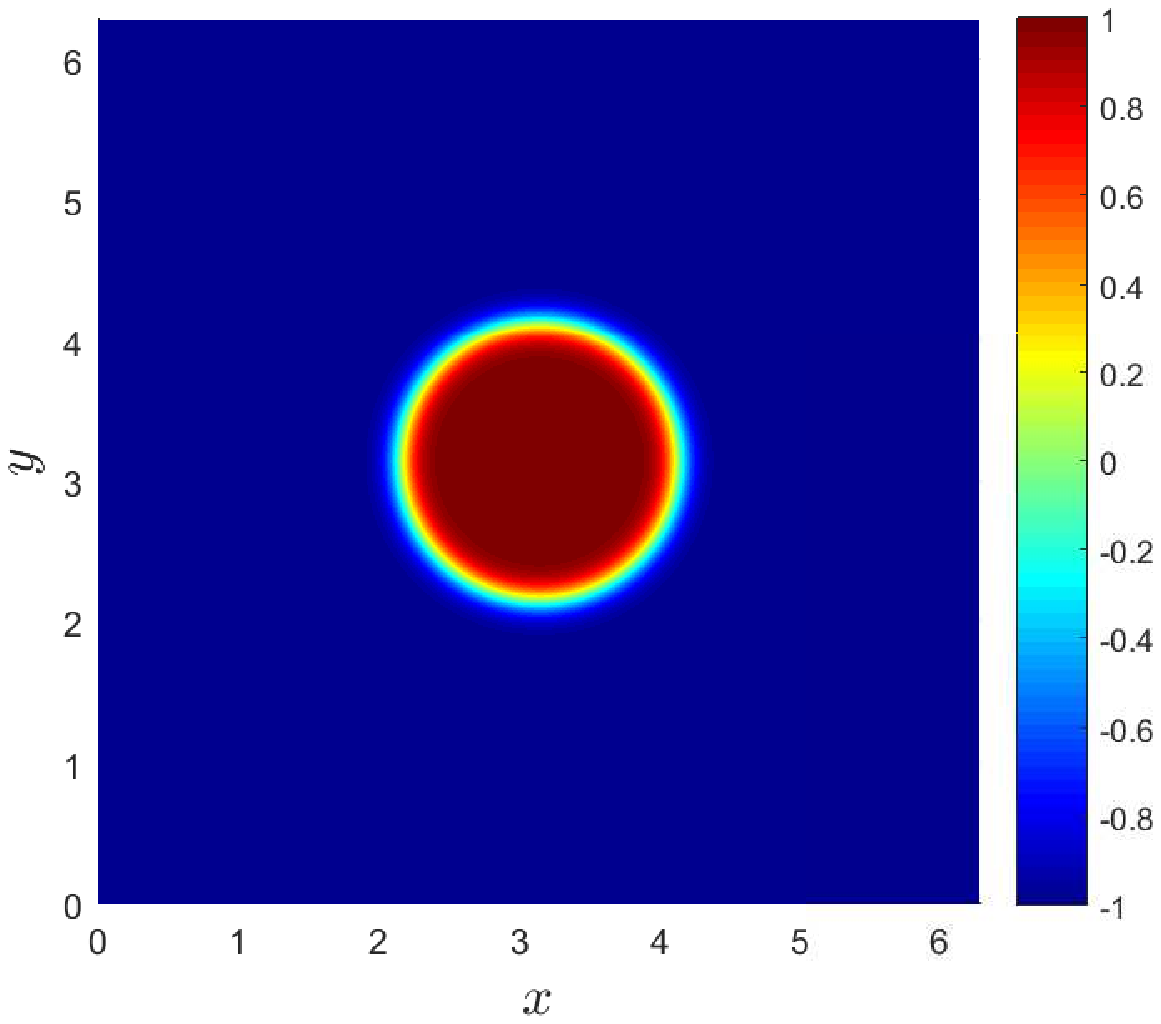}}
{\includegraphics[width=0.45\textwidth]{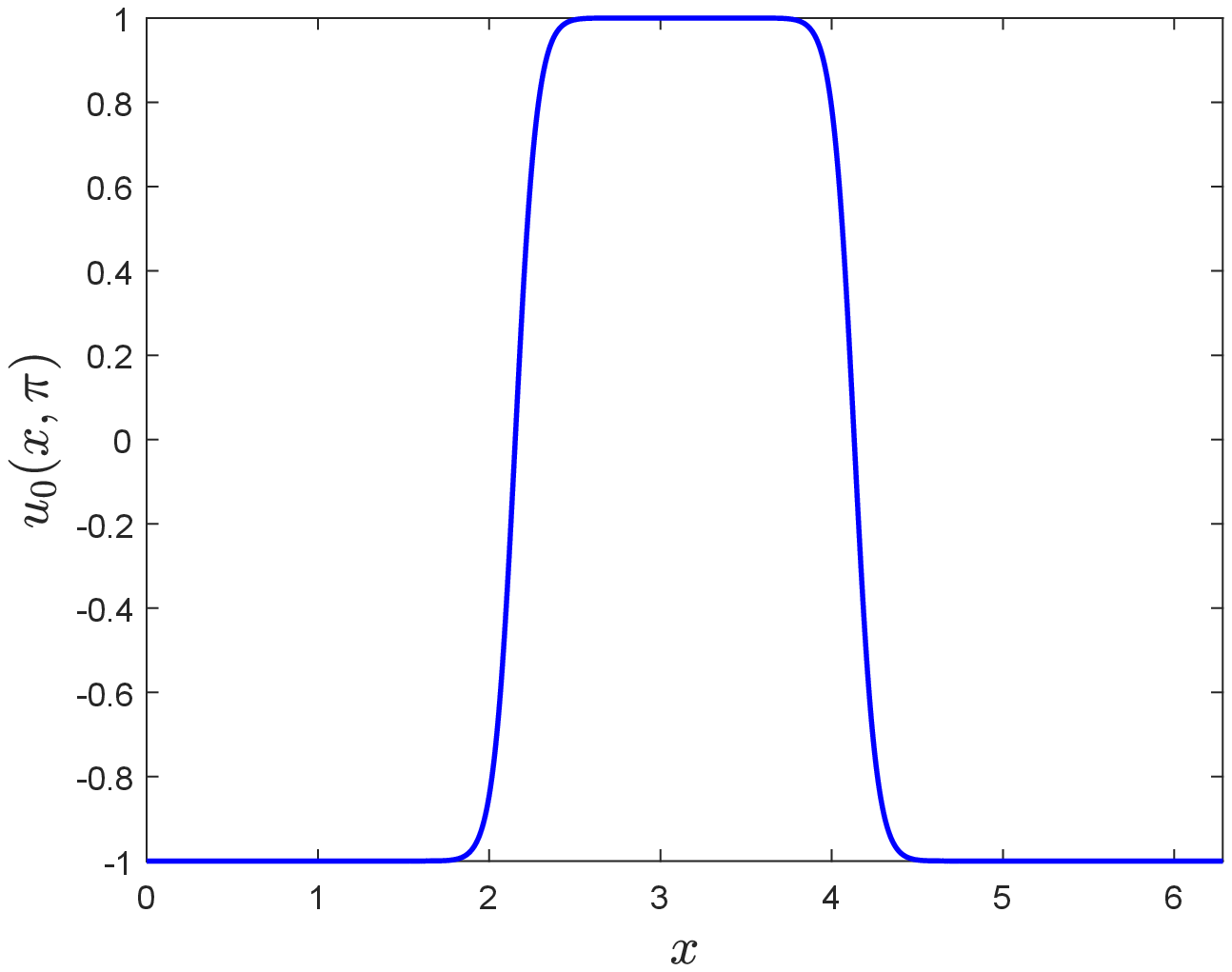}}}
\caption{Initial configuration of Example \ref{eg_bubble}.
Left: surface-project view; right: cross-section view at $y=\pi$.}
\label{fig_bubble_initial}
\end{figure}

This example is devoted to the relationship
between the discontinuities in the steady state solutions and the horizon parameter $\delta$.
Under the settings of the parameters given above, it is known from \eqref{steady_state_jump} that
the theoretical values of the jumps occurring at the discontinuity points can be formulated as
$$\text{Theoretical jump}=2\sqrt{1-\frac{0.12}{\delta^2}},\quad\delta>\delta_0=\sqrt{0.12}\approx0.3464.$$
We chose several $\delta$'s ($\delta=0.8,1.6,3.2$) larger than $\delta_0$ to observe the discontinuities
and the jumps in the numerical results,
and for the comparison, we also considered one case ($\delta=0.2$) with $\delta$ smaller than the critical value.
\tablename~\ref{table_bubble_jump} collects the theoretical and numerically computed jumps
occurring at the discontinuity points in the steady state solutions with various $\delta$'s.
It is observed that the numerical jumps match the theoretical values very well.

\begin{table}[!ht]
\centering
\caption{Theoretical and numerical jumps in steady state solutions in Example \ref{eg_bubble}.}
\label{table_bubble_jump}
\small
\begin{tabular}{|*{5}{c|}}
\hline
~ & $\delta=0.2$ & $\delta=0.8$ & $\delta=1.6$ & $\delta=3.2$ \\
\hline\hline
Theoretical jumps & 0 & 1.802776 & 1.952562 & 1.988247 \\
\hline
Numerical jumps & 0 & 1.804496 & 1.952713 & 1.988242 \\
\hline
\end{tabular}
\end{table}

\figurename~\ref{fig_bubbles} presents the evolutions of the bubble governed by the NAC equation with $\delta=0.2$ ($<\delta_0$), $\delta=0.8$ and $\delta=3.2$ (both $>\delta_0$), respectively.
In each row, the first three graphs give the surface-projection views  of the numerical solutions at several times
and the last graph cross-section views with $y=\pi$ by zooming-in around the interface.
For the case $\delta=0.2$, the bubble shrinks quickly and disappears finally,
which is similar to the process of the shrinkage occurring in the case of the LAC equation (see \cite{ChSh98}).
The evolutions for the cases $\delta=0.8$ and $\delta=3.2$ are similar:
the bubble does not shrink and the interface turns sharper and sharper so that
the solution preforms discontinuity on the interface after some times
and reaches the steady state with the expected jump.
It is seen from this example that
the NAC equation  with small $\delta$ has more similar dynamics with the local model,
which is consistent with the observations in Example \ref{eg_stability},
while the NAC equation  with large $\delta$, especially larger than $\delta_0$,
leads to the steady state solution within the discontinuity even though the initial state is smooth.

\begin{figure}[!ht]
\centering
\subfigure[$\delta=0.2$: at $t=1$, $40$, and $55$, and their cross-sections with $y=\pi$ and $x\in{[\frac{\pi}{2},\pi]}$]
{\includegraphics[width=0.22\textwidth]{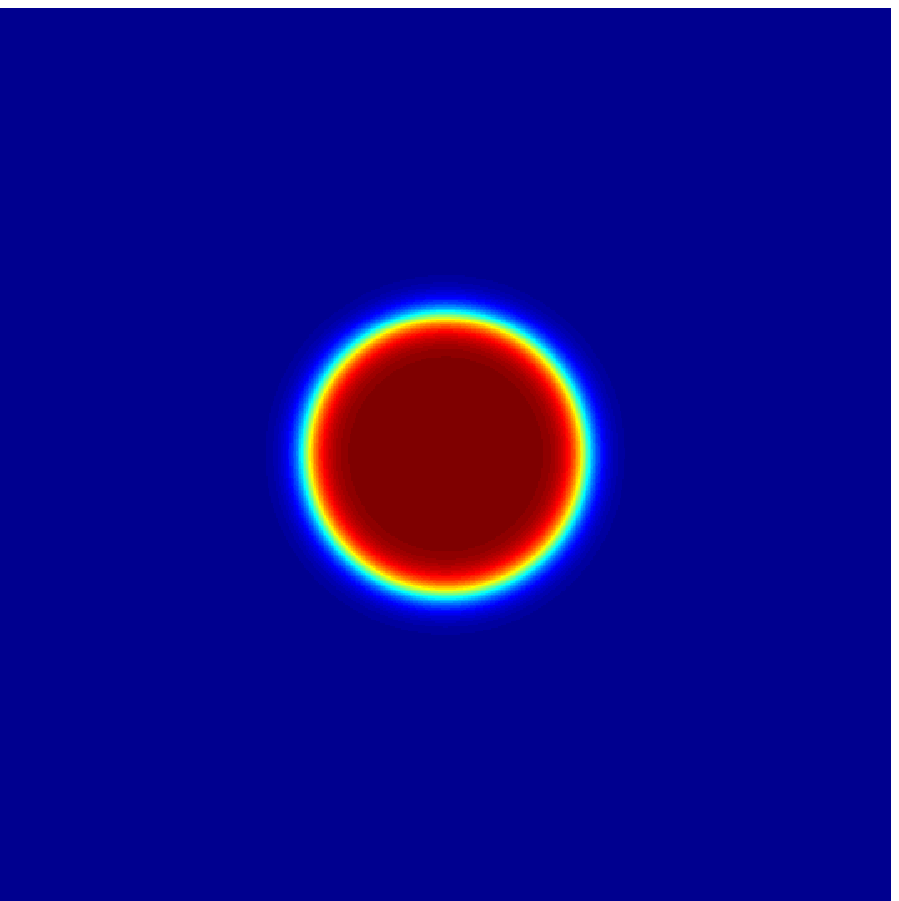}
\includegraphics[width=0.22\textwidth]{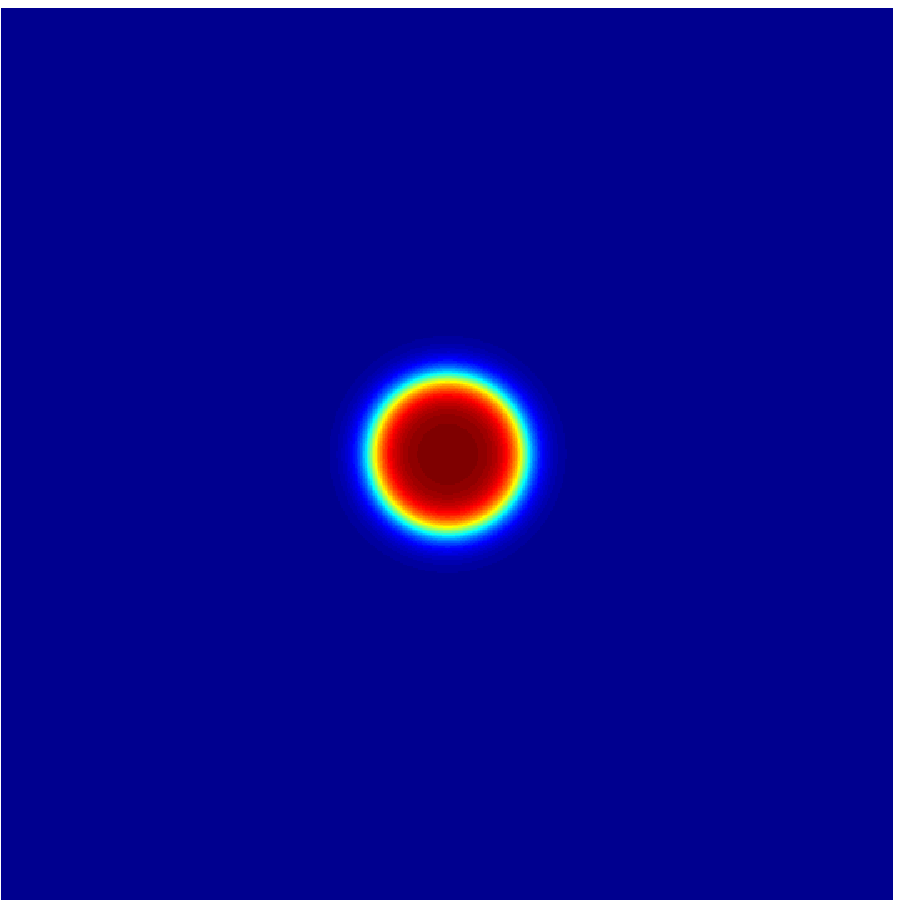}
\includegraphics[width=0.22\textwidth]{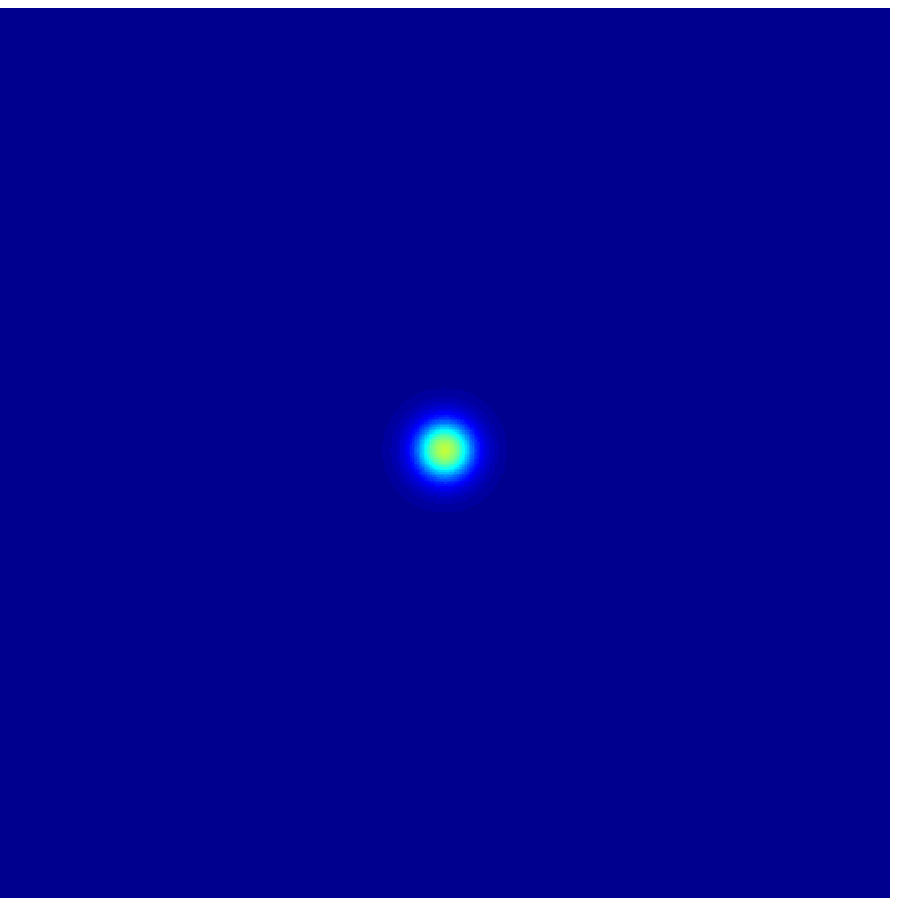}\hspace{0.2cm}
\includegraphics[width=0.3\textwidth]{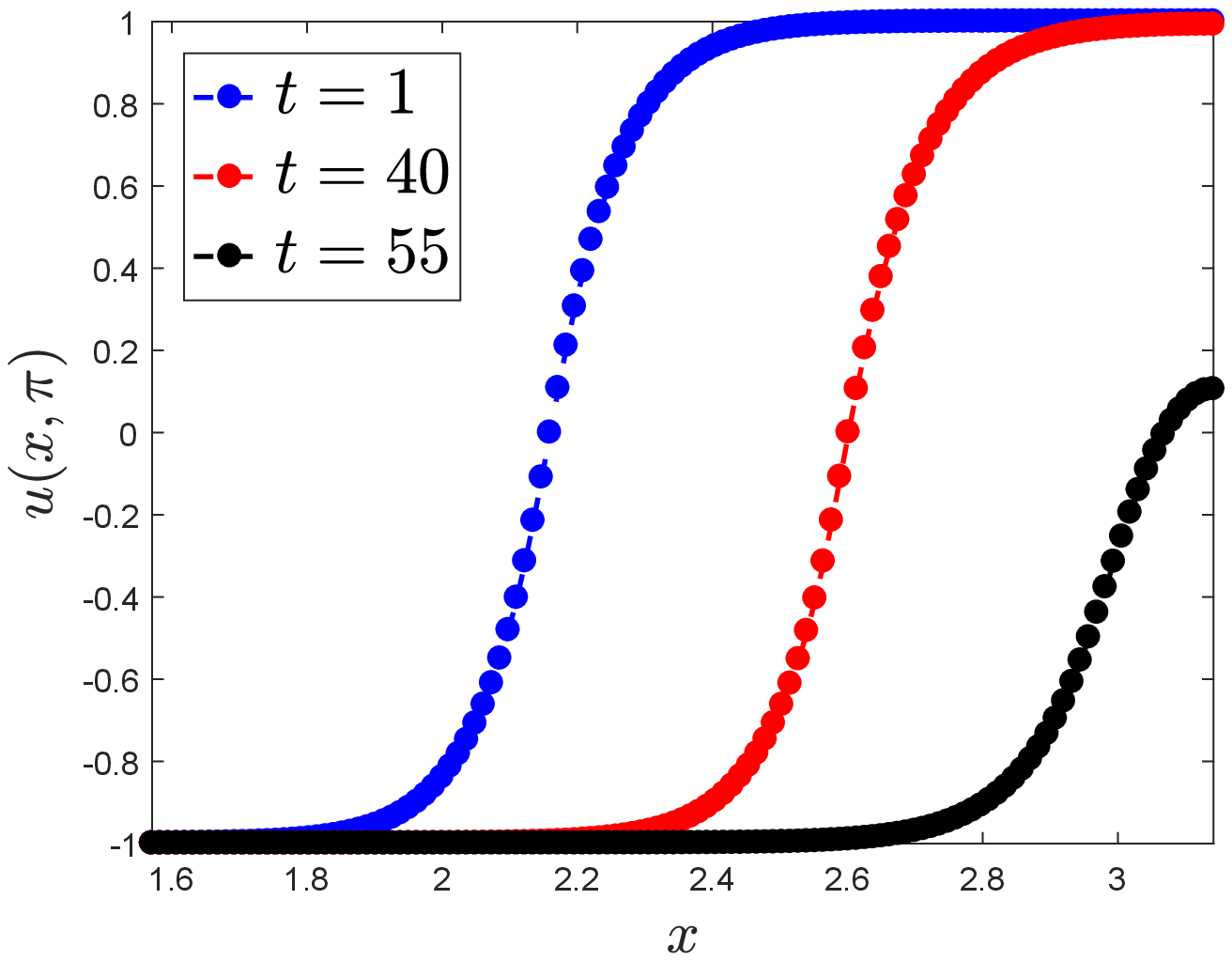}}
\subfigure[$\delta=0.8$:  at $t=1$, $3$, and $20$, and their cross-sections with $y=\pi$ and $x\in{[2.05,2.35]}$]
{\includegraphics[width=0.22\textwidth]{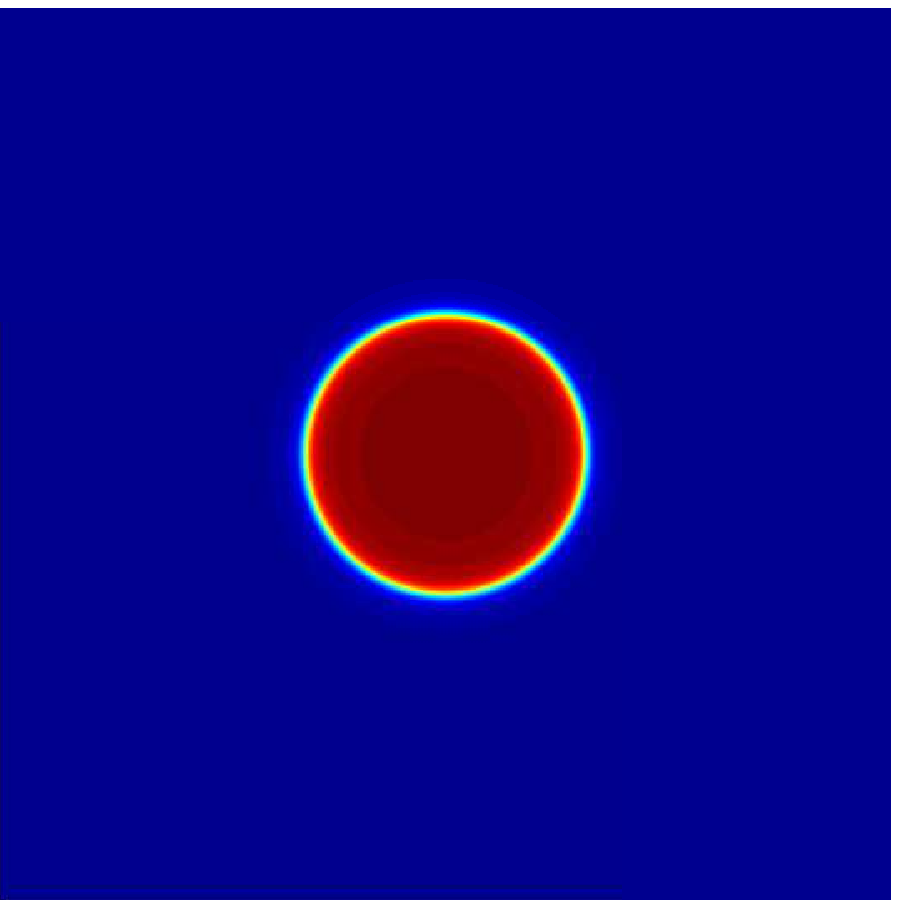}
\includegraphics[width=0.22\textwidth]{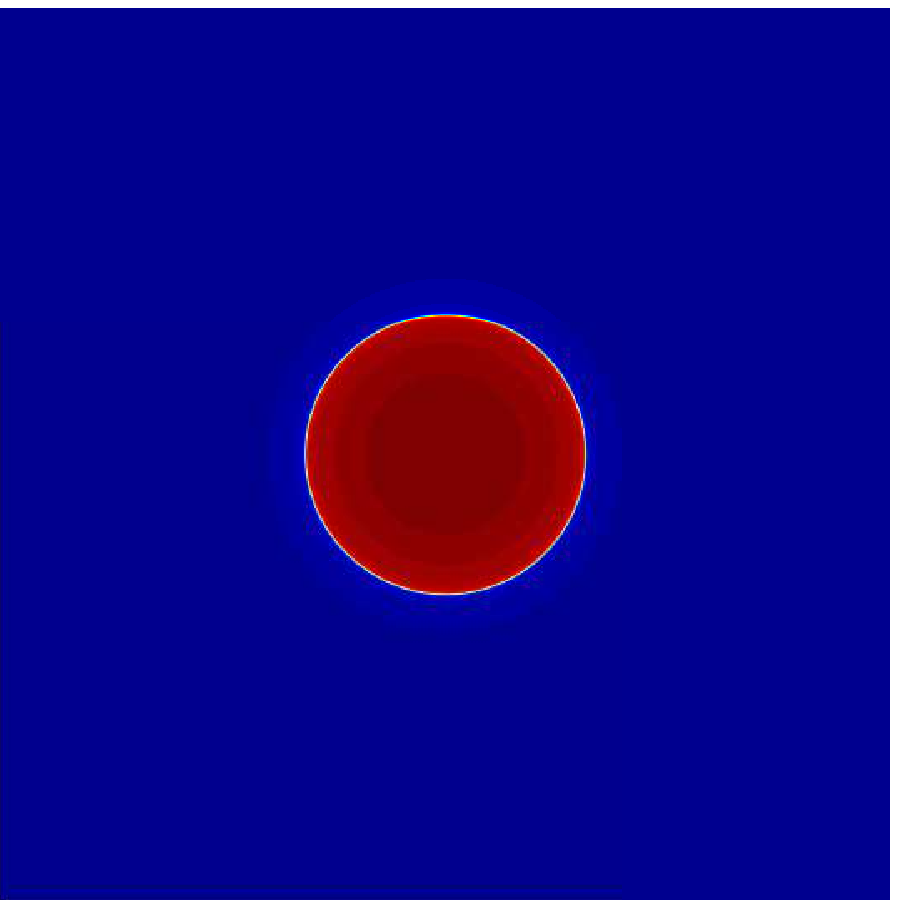}
\includegraphics[width=0.22\textwidth]{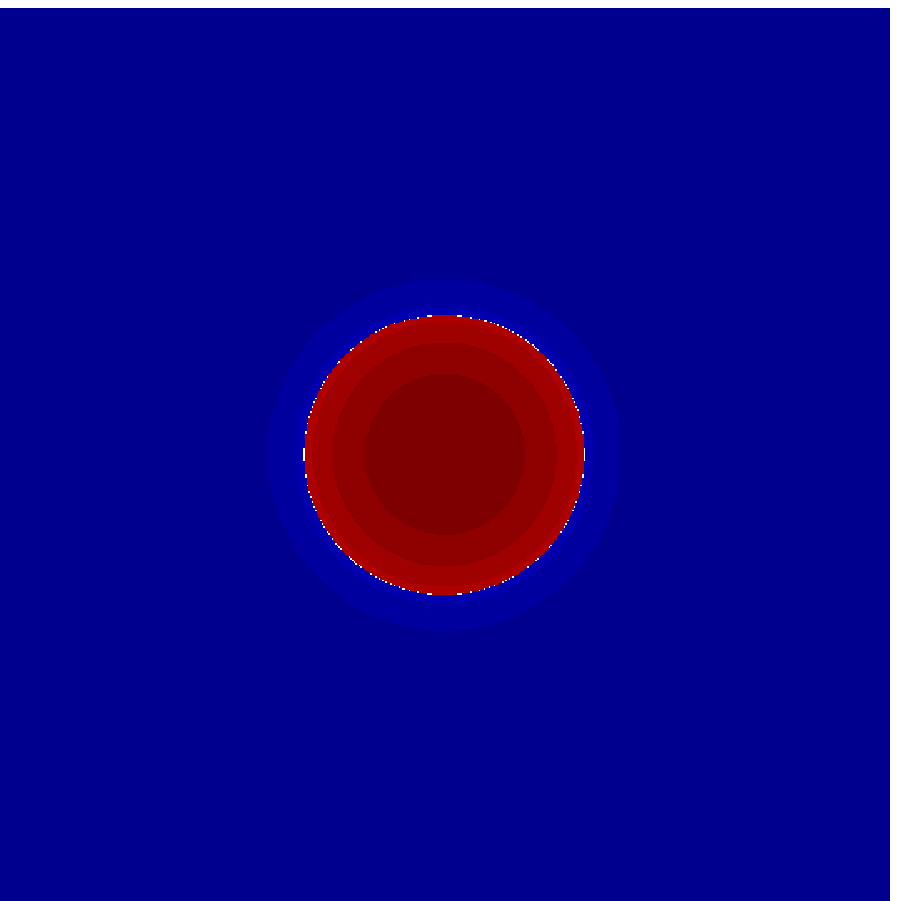}\hspace{0.2cm}
\includegraphics[width=0.3\textwidth]{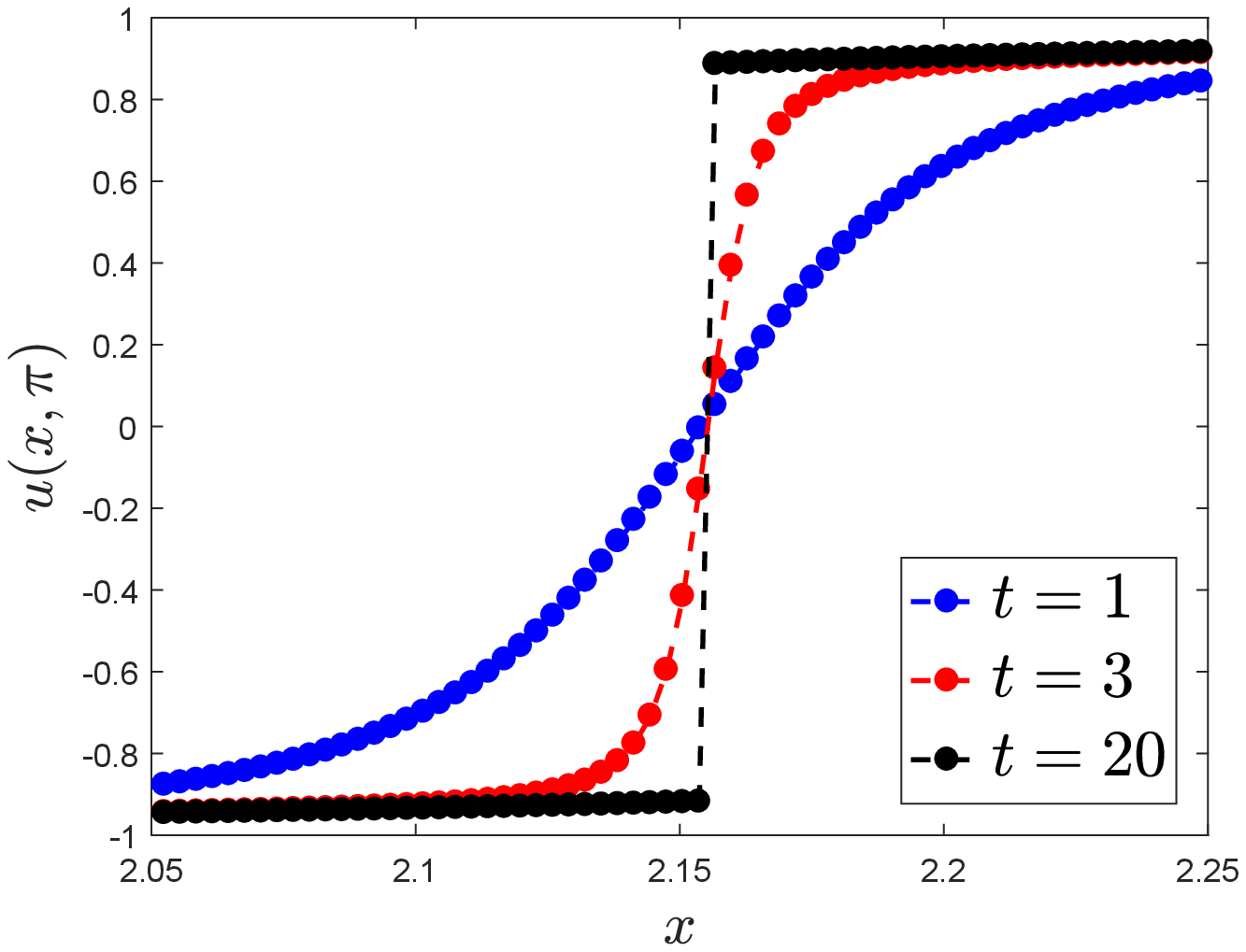}}
\subfigure[$\delta=3.2$:  at $t=1$, $3$, and $20$, and their cross-sections with $y=\pi$ and $x\in{[2.05,2.35]}$]
{\includegraphics[width=0.22\textwidth]{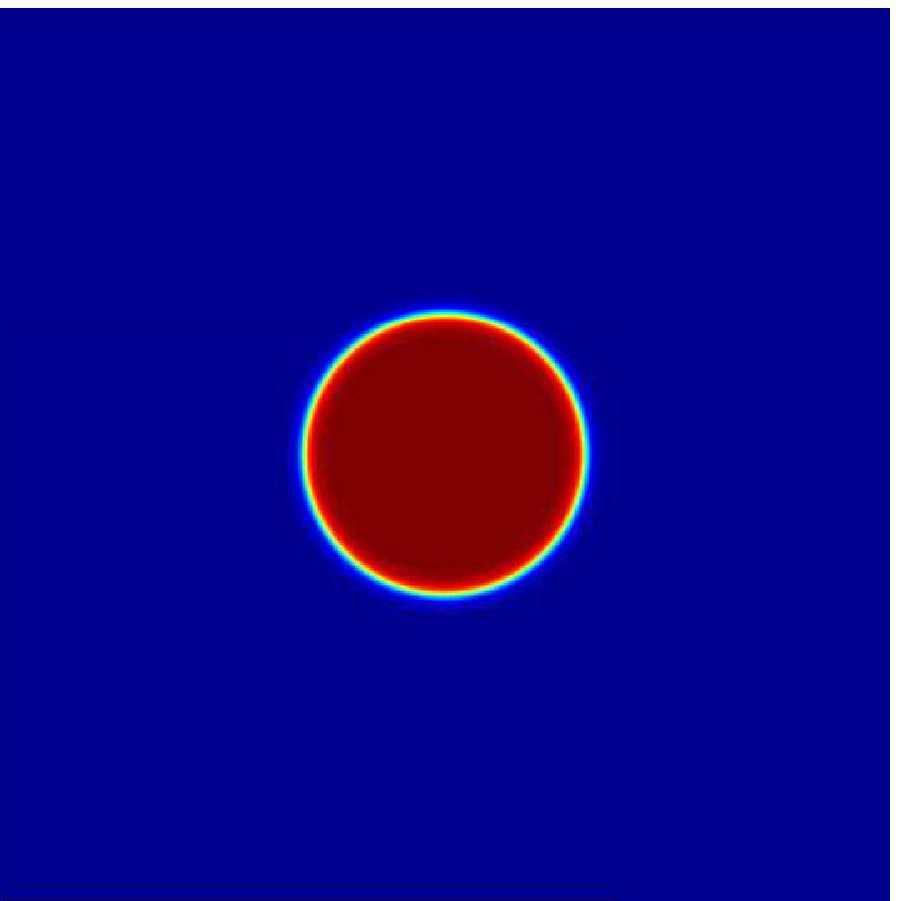}
\includegraphics[width=0.22\textwidth]{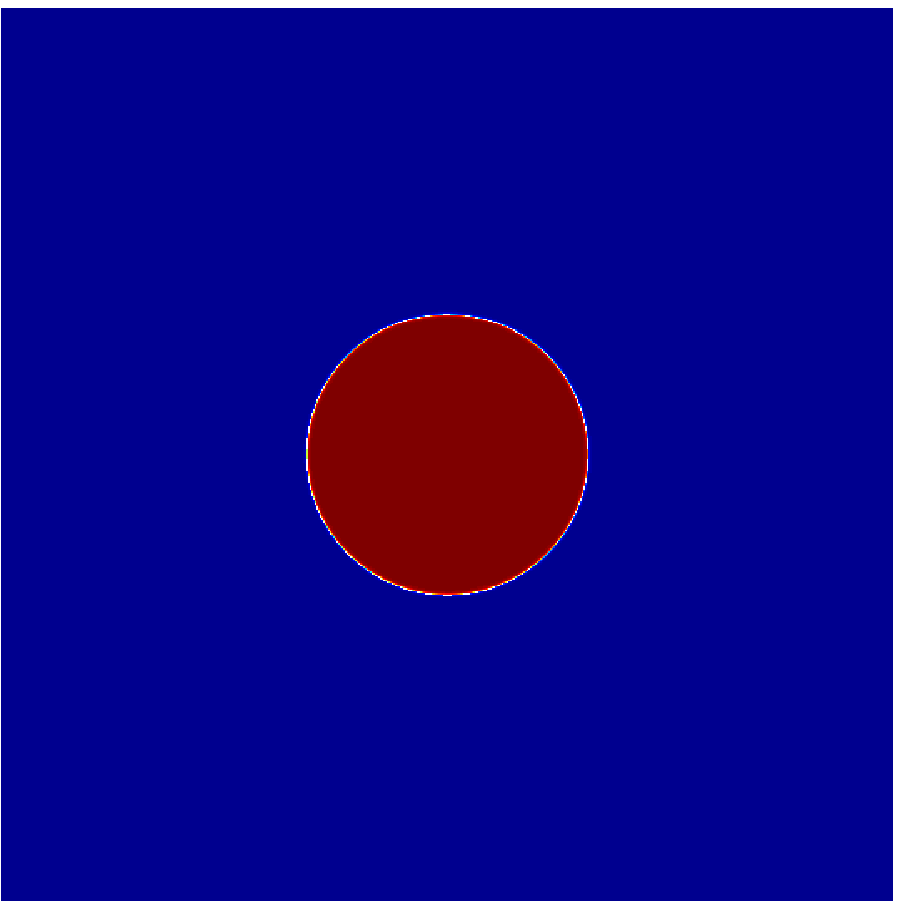}
\includegraphics[width=0.22\textwidth]{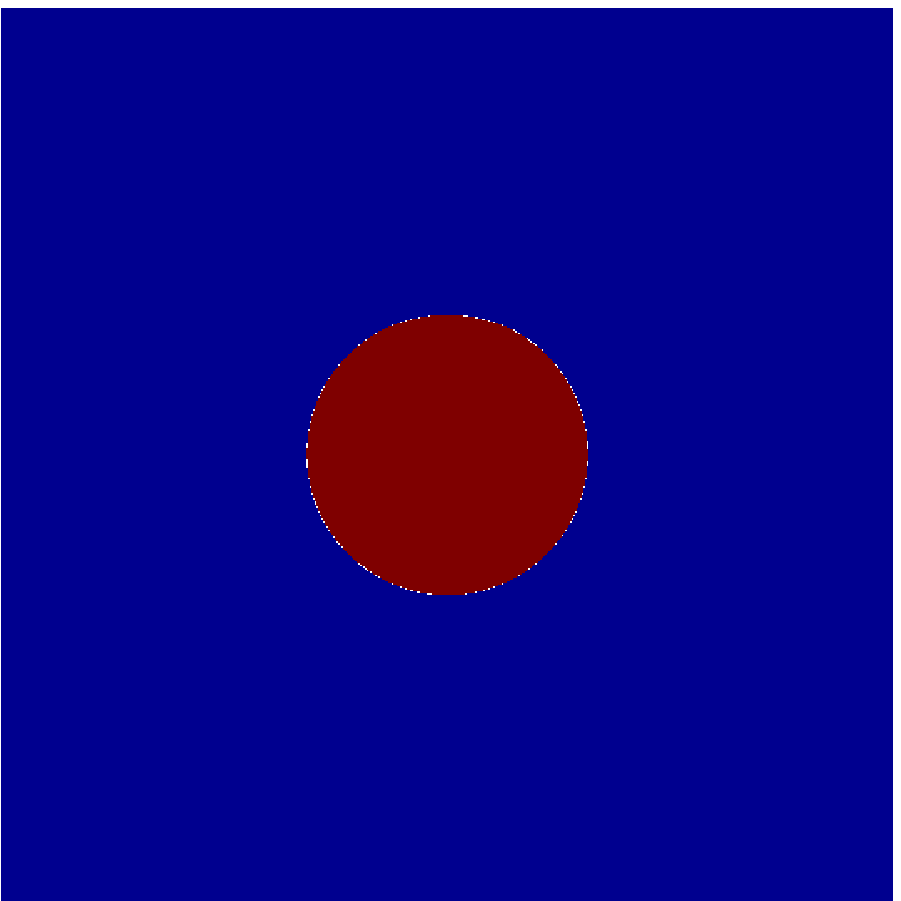}\hspace{0.2cm}
\includegraphics[width=0.3\textwidth]{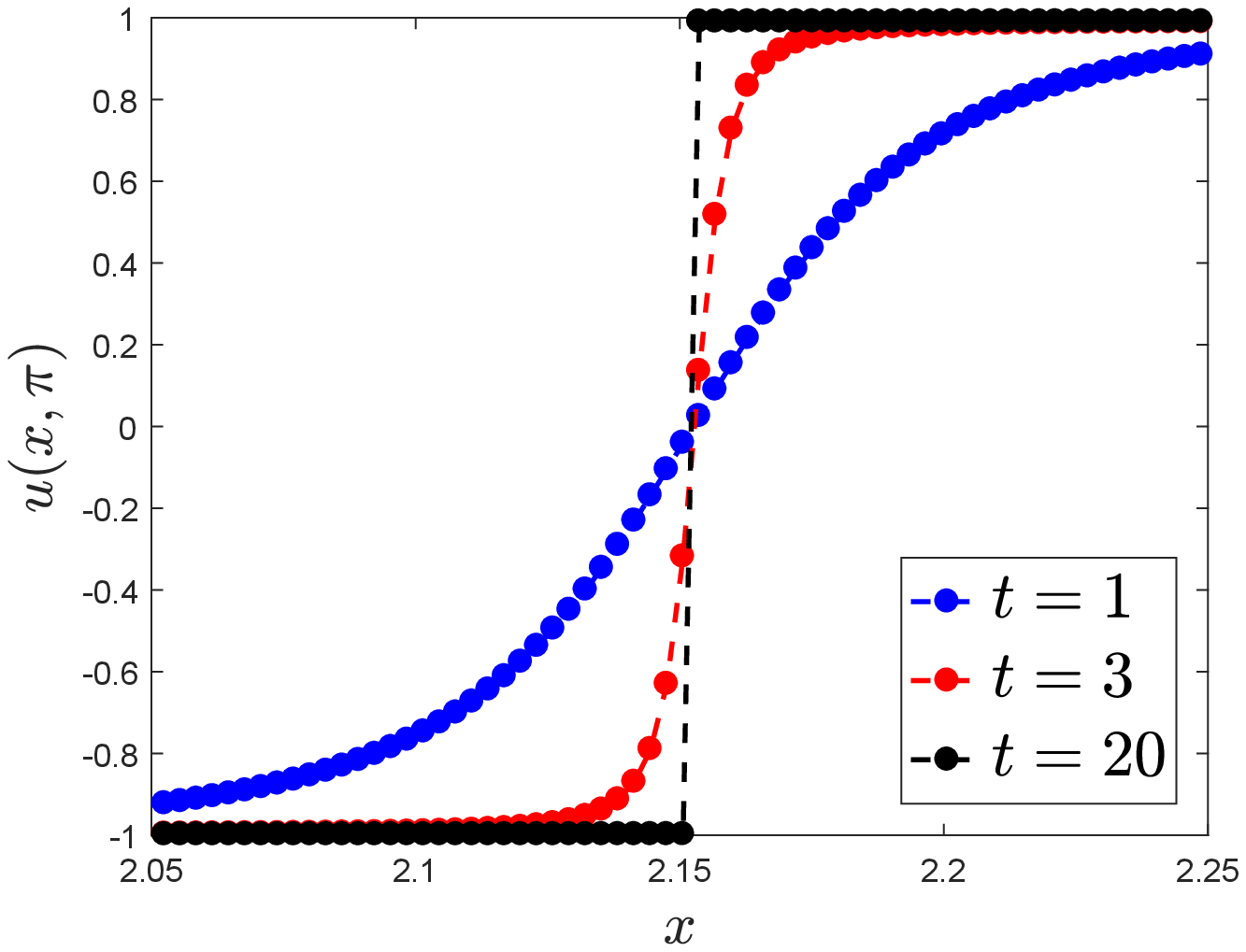}}
\caption{Evolutions of the bubble governed by the NAC equation with $\delta=0.2, 0.8, 3.2$ in  Example \ref{eg_bubble}.}
\label{fig_bubbles}
\end{figure}

\section{Conclusions}

We designed and analyzed  maximum principle preserving numerical schemes of for solving nonlocal Allen-Cahn equation
by using the quadrature-based finite difference method for spatial discretizations
and the exponential time differencing method for temporal integrations.
Especially, we developed the first order ETD and second order ETD Runge-Kutta schemes,
derive for both schemes the error estimates, and prove their energy stability as well as the asymptotic compatibility,
a special convergence considered for the numerical approximations of nonlocal models.
Numerical experiments are carried out to verify the theoretical results
and to study some more interesting properties of the solutions caused by the nonlocality.
The maximum principle preserving schemes studied here are up to the second order in time.
Whether higher order numerical schemes can preserve the maximum principle still remains open and is one of our future works.
In addition, for some other models, for instance, the nonlocal Cahn-Hilliard equation \cite{DuJuLiQi18,GuWaWi14},
the solution does not possess the maximum principle but is $L^\infty$ stable instead.
Numerical schemes naturally inheriting the $L^\infty$ stability, weaker than the maximum principle, are also worthy of study.


\end{document}